\documentclass[12pt]{ociamthesis}  

\usepackage[utf8]{inputenc}
\usepackage[T1]{fontenc}
\usepackage{amssymb}
\usepackage{libertinus}           
\usepackage{amsthm}
\usepackage{amsmath}
\usepackage{array}
\usepackage{mathbbol}
\usepackage{amsfonts}
\usepackage{hyperref}
\usepackage{bbold}
\usepackage{mathrsfs}
\usepackage{cleveref}
\usepackage{multicol}
\usepackage{tikz-cd}
\usepackage{xcolor}
\usepackage{float}
\usepackage{multirow}
\usepackage{booktabs}
\usepackage{enumitem} 

\setlist[itemize,enumerate]{itemsep=0em, topsep=0.3em}

\usepackage{tablefootnote}
\usepackage{chngcntr}

\usepackage{tikz}
\usepackage{caption}
\usetikzlibrary{arrows.meta}

\theoremstyle{plain}
\newtheorem{theorem}{Theorem}[chapter]
\newtheorem{lemma}[theorem]{Lemma}
\newtheorem{corollary}[theorem]{Corollary}
\newtheorem{axiom}[theorem]{Axiom}
\newtheorem{typee}[theorem]{Proposition}
\newtheorem{proposition}[theorem]{Proposition}

\newtheoremstyle{non-italic}
  {3pt} 
  {3pt} 
  {\normalfont} 
  {} 
  {\bfseries} 
  {.} 
  { } 
  {} 
\theoremstyle{non-italic}
\newtheorem{example}[theorem]{Example}
\newtheorem{definition}[theorem]{Definition}

\newtheoremstyle{remark}
  {3pt} 
  {3pt} 
  {\normalfont} 
  {} 
  {\itshape} 
  {.} 
  { } 
  {} 
\theoremstyle{remark}
\newtheorem{remark}[theorem]{Remark}

\renewcommand{\texttt}[1]{\(\mathtt{#1}\)}

\newcommand{\Type}{\mathtt{Type}}
\newcommand{\pr}{\mathtt{pr}}

\newcommand{\unt}{\mathbb{1}}

\newcommand{\redDiamond}{\,\textcolor{red}{$\blacksquare$}\,}

\newcommand{\ct}{%
  \mathchoice{\mathbin{\raisebox{0.5ex}{$\displaystyle\centerdot$}}}%
             {\mathbin{\raisebox{0.5ex}{$\centerdot$}}}%
             {\mathbin{\raisebox{0.25ex}{$\scriptstyle\,\centerdot\,$}}}%
             {\mathbin{\raisebox{0.1ex}{$\scriptscriptstyle\,\centerdot\,$}}}
}

\definecolor{customgreen}{RGB}{60,179,113}
\definecolor{customred}{RGB}{205,92,92}
\definecolor{customblue}{HTML}{0074D9}

\usepackage{tikz}
\usepackage{xcolor}

\definecolor{customgreen}{RGB}{60,179,113}
\definecolor{customblue}{HTML}{0074D9}
\definecolor{customred}{RGB}{205,92,92}

\newcommand{\figOne}{
\begin{tikzpicture}[
    label/.style={rectangle, inner sep=4pt, outer sep=0pt},
]

\draw[gray] (0,0) circle (1.2cm);
\draw[gray] (5,0) ellipse (3cm and 2.1cm);
\draw[gray, dashed] (5,1) ellipse (1.7cm and 0.8cm);
\draw[gray, dashed] (5,-0.95) ellipse (1.7cm and 0.8cm);

\fill (0,0.5) circle (2.1pt) node[label, above] {$x$};
\fill (0,-0.5) circle (2.1pt) node[label, below] {$y$};
\fill (4,0.8) circle (2.1pt) node[label, above] {$f(x)$};
\fill (4,-0.8) circle (2.1pt) node[label, below] {$f(y)$};
\fill (6,0.8) circle (2.1pt) node[label, above] {$g(x)$};
\fill (6,-0.8) circle (2.1pt) node[label, below] {$g(y)$};

\draw[-{Stealth[scale=1.5]}] (0,0.5) -- node[left=2pt] {$p$} (0,-0.5);
\draw[-{Stealth[scale=1.5]}] (4,0.8) -- node[left=2pt] {$\mathtt{ap}_f(p)$} (4,-0.8);
\draw[-{Stealth[scale=1.5]}] (6,0.8) -- node[right=2pt] {$\mathtt{ap}_g(p)$} (6,-0.8);
\draw[-{Stealth[scale=1.5]}] (4,0.8) -- node[above=2pt] {$u$} (6,0.8);
\draw[-{Stealth[scale=1.5]}] (4,-0.8) -- node[above=2pt] {$p_*(u)$} (6,-0.8);

\node[anchor=east] at (-1.2,0) {$A$};
\node[anchor=west] at (8,0) {$B$};

\end{tikzpicture}
}

\newcommand{\figTrans}{
\begin{tikzpicture}[
    label/.style={rectangle, inner sep=4pt, outer sep=0pt},
]

\draw[gray] (0,0) circle (1.2cm);
\draw[gray] (4,2) circle (0.8cm);
\draw[gray] (4,-1.3) circle (1.6cm);

\fill (0,0.5) circle (2.1pt) node[label, above] {$x$};
\fill (0,-0.5) circle (2.1pt) node[label, below] {$y$};
\fill (4,2) circle (2.1pt) node[label, above] {$f(x)$};
\fill (4,-1.9) circle (2.1pt) node[label, below] {$f(y)$};
\fill (4,-0.7) circle (2.1pt) node[label, above] {$p_*(f(x))$};

\draw[-{Stealth[scale=1.5]}] (0,0.5) -- node[left=2pt] {$p$} (0,-0.5);
\draw[-{Stealth[scale=1.5]}] (4,-0.7) -- node[left=1pt] {$\mathtt{apd}_f(p)$} (4,-1.9);
\draw[-{Straight Barb[scale=1.5]}] (4,1.2) -- node[right=2pt] {$p_* \equiv \mathtt{transport}^P(p)$} (4,0.3);

\node[label, anchor=south] at (0,1.2) {$A$};
\node[label, anchor=west] at (4.8,2) {$P(x)$};
\node[label, anchor=west] at (5.6,-1.3) {$P(y)$};

\end{tikzpicture}
}

\newcommand{\figTransFun}{
\begin{tikzpicture}[
    label/.style={rectangle, inner sep=3pt, outer sep=0pt},
]

\draw[gray] (-1,0) circle (1.2cm);
\draw[gray] (3.7,1) circle (0.5cm);
\draw[gray] (6.3,1) circle (0.5cm);
\draw[gray, dashed] (5,1.2) ellipse (2.5cm and 1.1cm);
\draw[gray, dashed] (5,-1.2) ellipse (2.5cm and 1.1cm);
\draw[gray] (3.7,-1) circle (0.5cm);
\draw[gray] (6.3,-1) circle (0.5cm);

\fill (-1,0.5) circle (2.1pt) node[label, above] {$x_1$};
\fill (-1,-0.5) circle (2.1pt) node[label, below] {$x_2$};

\draw[-{Stealth[scale=1.5]}] (-1,0.5) -- node[left=2pt] {$p$} (-1,-0.5);
\draw[-{Straight Barb[scale=1.5]}] (3.7,0.5) -- node[left=2pt] {$\mathtt{transport}^A(p)$} (3.7,-0.5);
\draw[-{Straight Barb[scale=1.5]}] (6.3,0.5) -- node[right=2pt] {$\mathtt{transport}^B(p)$} (6.3,-0.5);
\draw[-{Straight Barb[scale=1.5]}] (4.2,1) -- node[above=2pt] {$f$} (5.8,1);
\draw[-{Straight Barb[scale=1.5]}] (4.2,-1) -- node[above=2pt] {$p_*(f)$} (5.8,-1);

\node[label, anchor=south] at (-1,1.2) {$X$};
\node[label, anchor=south] at (3.7,1.5) {$A(x_1)$};
\node[label, anchor=north] at (3.7,-1.5) {$A(x_2)$};
\node[label, anchor=south] at (6.3,1.5) {$B(x_1)$};
\node[label, anchor=north] at (6.3,-1.5) {$B(x_2)$};

\end{tikzpicture}
}

\newcommand{\figTwo}{
\begin{tikzpicture}[
    label/.style={rectangle, inner sep=4pt, outer sep=0pt},
]

\draw[gray] (0,0) circle (1.2cm);
\draw[gray] (4,0) circle (2.1cm);

\fill (0,0.5) circle (2.1pt) node[label, above] {$x$};
\fill (0,-0.5) circle (2.1pt) node[label, below] {$y$};
\fill (4.5,0) circle (2.1pt) node[label, right] {$f(x)$};
\fill (4.5,1.2) circle (2.1pt) node[label, above] {$p_*(f(x))$};
\fill (4.5,-1.2) circle (2.1pt) node[label, below] {$f(y)$};

\draw[-{Stealth[scale=1.5]}] (0,0.5) -- node[left=2pt] {$p$} (0,-0.5);
\draw[-{Stealth[scale=1.5]}] (4.5,0) -- node[right=2pt] {$\mathtt{ap}_f(p)$} (4.5,-1.2);
\draw[-{Stealth[scale=1.5]}] (4.5,1.2) -- node[right=2pt] {$\mathtt{tr\text{-}const}^B_p(f(x))$} (4.5,0);
\draw[-{Stealth[scale=1.5]}] (4.5,1.2) to[out=-180,in=180] node[left=2pt] {$\mathtt{apd}_f(p)$} (4.5,-1.2);

\draw[-{Straight Barb[scale=1.5]}] (1.2,0) -- node[below=2pt] {$f$} (1.9,0);

\draw[-{Straight Barb[scale=1.5]}] (6.1,0) .. controls (8.1,1.2) and (8.1,-1.2) .. (6.1,0) node[midway, right=2pt] {$p_*$};

\node[label, anchor=south] at (0,1.2) {$A$};
\node[label, anchor=south] at (4,2.1) {$B$};

\end{tikzpicture}
}

\newcommand{\figTransPath}{
\begin{tikzpicture}[
    label/.style={rectangle, inner sep=4pt, outer sep=0pt},
]

\draw[gray] (0,0) circle (2cm);

\fill (-0.6,1) circle (2.1pt) node[label, above] {$x_1$};
\fill (-0.6,-1) circle (2.1pt) node[label, below] {$x_2$};
\fill (1,0) circle (2.1pt) node[label, right] {$a$};

\draw[-{Stealth[scale=1.5]}] (-0.6,1) -- node[left=4pt] {$p$} (-0.6,-1);
\draw[-{Stealth[scale=1.5]}] (1,0) -- node[above right=2pt] {$q_1$} (-0.6,1);
\draw[-{Stealth[scale=1.5]}] (1,0) -- node[below right=2pt] {$p_*(q_1)$} (-0.6,-1);

\begin{scope}[rotate=60]
    \draw[gray, dashed] (0.5,0.2) ellipse (0.7cm and 1.6cm);
\end{scope}

\begin{scope}[rotate=-60]
    \draw[gray, dashed] (0.8,0.1) ellipse (0.7cm and 1.6cm);
\end{scope}

\node[label, anchor=east] at (-2,0) {$A$};

\end{tikzpicture}
}

\newcommand{\figPoOne}{
    \begin{tikzpicture}[
        label/.style={rectangle, inner sep=4pt, outer sep=0pt},
        point/.style={circle, inner sep=0.5mm, fill=black},
        circ/.style={draw=gray, fill=none}
    ]
    
    \draw[circ] (0,0) circle (0.8cm); 
    \draw[circ] (-2.8,0) circle (1.1cm); 
    \draw[circ] (2.8,0) circle (1.1cm); 
    \draw[circ, dashed] (-2.8,-3) circle (1.1cm); 
    \draw[circ, dashed] (2.8,-3) circle (1.1cm); 
    \draw[circ] (0,-3) ellipse (5cm and 1.5cm); 
    
    \fill[black] (0,0) circle (2.1pt) node[label, above] {$c$}; 
    \fill[black] (-2.6,-0.3) circle (2.1pt) node[label, above] {$f(c)$}; 
    \fill[black] (2.7,-0.4) circle (2.1pt) node[label, above] {$g(c)$}; 
    \fill[black] (-2.6,-3.3) circle (2.1pt) node[label, above] {$\mathtt{inl}(f(c))$}; 
    \fill[black] (2.7,-3.4) circle (2.1pt) node[label, above] {$\mathtt{inr}(g(c))$}; 
    
    \draw[-{Stealth[scale=1.5]}] (-2.6,-3.3) .. controls (0,-4) .. (2.7,-3.4) node[midway, below] {$\mathtt{glue}(c)$}; 
    \draw[-{Straight Barb[scale=1.5]}] (-0.8,0) -- (-1.7,0) node[midway, above] {$f$}; 
    \draw[-{Straight Barb[scale=1.5]}] (0.8,0) -- (1.7,0) node[midway, above] {$g$}; 
    \draw[-{Straight Barb[scale=1.5]}] (-2.8,-1.1) -- (-2.8,-1.75) node[midway, left] {$\mathtt{inl}$}; 
    \draw[-{Straight Barb[scale=1.5]}] (2.8,-1.1) -- (2.8,-1.75) node[midway, right] {$\mathtt{inr}$}; 
    \draw[-{Straight Barb[scale=1.5]}] (0,-0.8) -- (0,-1.5) node[midway, right] {$\mathtt{glue}$}; 
    
    \node[anchor=east] at (-3.9,0) {$A$};
    \node[anchor=north] at (0,1.3) {$C$};
    \node[anchor=west] at (3.9,0) {$B$};
    \node[anchor=north] at (0,-1.6) {$A \sqcup^{C} B$};
    
    \end{tikzpicture}
}

\newcommand{\figFourOneSeven}{
    \begin{tikzpicture}[
        label/.style={rectangle, inner sep=4pt, outer sep=0pt},
        point/.style={circle, inner sep=0.5mm, fill=black},
        circ/.style={draw=gray, fill=none}
    ]
    
    \draw[circ] (-1,0) circle (1.2cm); 
    \draw[circ] (4,0) ellipse (2.32cm and 2.25cm); 
    
    \fill[black] (-1,0.4) circle (2.1pt) node[label, above] {$x$}; 
    \fill[black] (-1,-0.4) circle (2.1pt) node[label, below] {$\star_{X}$}; 
    \fill[black] (3,1.1) circle (2.1pt) node[label, above] {$f(x)$}; 
    \fill[black] (3,-1.1) circle (2.1pt) node[label, below] {$f(\star_{X})$}; 
    \fill[black] (4,0) circle (2.1pt) node[label, right] {$\star_{Y}$}; 
    
    \draw[-{Stealth[scale=1.5]}] (-1,0.4) -- (-1,-0.4) node[midway, left] {$q$}; 
    \draw[-{Stealth[scale=1.5]}] (3,-1.1) -- (3,1.1) node[midway, left, yshift=10px] {$\mathtt{ap}_f(q^{-1})$}; 
    \draw[-{Stealth[scale=1.5]}] (3,1.1) -- (4,0) node[midway, above right] {$p$}; 
    \draw[-{Stealth[scale=1.5]}] (4,0) -- (3,-1.1)   node[midway, below right] {$f_*^{-1}$}; 
    \draw[-{Stealth[scale=1.5]}] (4,0) .. controls (6,3) and (6,-3) .. (4,0) node[midway, right] {$r$}; 
    
    \node[anchor=east] at (-2.2,0) {$X$};
    \node[anchor=west] at (6.5,0) {$Y$};
    
    \draw[-{Straight Barb[scale=1.5]}] (0.2,0) -- (1.68,0) node[midway, below] {$f$}; 
    
    \end{tikzpicture}
}

\newcommand{\figPoTwo}{
    \begin{tikzpicture}[
        label/.style={rectangle, inner sep=4pt, outer sep=0pt},
        point/.style={circle, inner sep=0.5mm, fill=black},
        circ/.style={draw=gray, fill=none}
    ]
    
    \begin{scope}[rotate=45]
        \draw[circ] (-2.8,0) ellipse (2.2cm and 0.4cm); 
    \end{scope}
    \begin{scope}[rotate=-45]
        \draw[circ] (2.8,0) ellipse (2.2cm and 0.4cm); 
    \end{scope}
    
    \draw[circ] (0,-2.7) circle (0.6cm); 
    \draw[circ] (0,-5.5) ellipse (5cm and 1.5cm); 
    
    \draw[circ, dashed] (-2.8,-5.5) circle (1.1cm); 
    \draw[circ, dashed] (2.8,-5.5) circle (1.1cm); 
    
    \fill[black] (-1.4,-1.4) circle (2.1pt) node[label, below left=2pt] {$x$}; 
    \fill[black] (1.4,-1.4)circle (2.1pt) node[label, below right=2pt] {$y$}; 
    \fill[black] (0.1,-2.5) circle (2.1pt) node[label, below] {$q$}; 
    \fill[black] (-2.3,-5.8) circle (2.1pt) node[label, left] {$\mathtt{inl}(x)$}; 
    \fill[black] (2.3,-5.8) circle (2.1pt) node[label, right] {$\mathtt{inr}(y)$}; 
    
    \node[anchor=south east] at (-2.8,-1.8) {$X$}; 
    \node[anchor=south west] at (2.8,-1.8) {$Y$}; 
    \node[anchor=south] at (0,-2.2) {$Q(x)(y)$};
    \node[anchor=north] at (0,-4) {$X \sqcup^{Q} Y$};
    
    \draw[-{Stealth[scale=1.5]}] (-2.3,-5.8) .. controls (0,-6.5) .. (2.3,-5.8) node[midway, above=5pt] {$\mathtt{glue}_{x, y} (q) $}; 
    
    \draw[-{Straight Barb[scale=1.5]}] (-2.8,-3.2) -- (-2.8,-4.25) node[midway, left=2pt] {$\mathtt{inl}$}; 
    \draw[-{Straight Barb[scale=1.5]}] (2.8,-3.2) -- (2.8,-4.25) node[midway, right=2pt] {$\mathtt{inr}$}; 
    \draw[-{Straight Barb[scale=1.5]}] (0,-3.3) -- (0,-4) node[midway, right] {$\mathtt{glue}_{x, y}$}; 
    
    \draw[-{Straight Barb[scale=1.5]}, gray, dashed] (-1.1,-1.7) -- (-0.5,-2.4); 
    \draw[-{Straight Barb[scale=1.5]}, gray, dashed] (1.1,-1.7) -- (0.5,-2.4); 
    
    \end{tikzpicture}
}

\newcommand{\figCode}{
    \begin{tikzpicture}[
        label/.style={rectangle, inner sep=4pt, outer sep=0pt},
        point/.style={circle, inner sep=0.5mm, fill=black},
        circ/.style={draw=gray, fill=none},
        bluearrow/.style={draw=customblue, -{Stealth[scale=1.5]}, thick},
        redarrow/.style={draw=customred, -{Stealth[scale=1.5]}, thick},
        grayarrow/.style={draw=gray, -{Stealth[scale=1.5]}, thick}
    ]
    
    \draw[circ] (0,0) ellipse (5cm and 2cm); 
    \node[anchor=south] at (0,2.1) {$X \sqcup^{Q} Y$};
    
    \draw[circ, dashed] (-2.8,0) circle (1.5cm); 
    \draw[circ, dashed] (2.8,0) circle (1.5cm); 
    
    \fill[black] (-2.8,0.7) circle (2.1pt) node[label, above] {$\mathtt{inl}(x_0)$}; 
    \fill[black] (-2.8,-0.7) circle (2.1pt) node[label, below] {$\mathtt{inl}(x_1)$}; 
    \fill[black] (2.8,0.7) circle (2.1pt) node[label, above] {$\mathtt{inr}(y_0)$}; 
    \fill[black] (2.8,-0.7) circle (2.1pt) node[label, below] {$\mathtt{inr}(y_1)$}; 
    
    \draw[bluearrow] (-2.8,0.7) .. controls (0,1.5) .. (2.8,0.7) node[midway, above, text=customblue] {$\mathtt{glue}(q_{00})$}; 
    \draw[bluearrow] (2.8,0.7) .. controls (0,1) .. (-2.8,-0.7) node[midway, below, xshift=10pt, text=customblue] {$\mathtt{glue}(q_{10})^{-1}$}; 
    \draw[redarrow] (-2.8,0.7) .. controls (0,-1) .. (2.8,-0.7) node[midway, above, xshift=10pt, text={rgb,255:red,205; green,92; blue,92}] {$\mathtt{glue}(q_{01})$}; 
    \draw[-{Stealth[scale=1.5]}] (-2.8,-0.7) .. controls (0,-1.5) .. (2.8,-0.7) node[midway, below] {$\mathtt{glue}(q_{11})$}; 
    
    \draw[grayarrow] (-2.8,0.7) .. controls (0,0.9) .. (2.8,-0.7) node[midway, above, xshift=-10pt, text=gray] {$r$}; 
    \draw[grayarrow] (-2.8,0.7) .. controls (-3.5,0) .. (-2.8,-0.7) node[midway, left, text=gray] {$s$}; 
    
    \end{tikzpicture}
}

\newcommand{\figAbc}{
    \begin{tikzpicture}[
        label/.style={rectangle, inner sep=4pt, outer sep=0pt},
    ]
    
    \draw[gray] (-1,0) circle (1.2cm);
    \draw[gray] (3.7,1) circle (0.5cm);
    \draw[gray] (7,0) ellipse (1cm and 2cm);
    \draw[gray] (7,1.5) circle (0.2cm);
    \draw[gray] (7,0.6) circle (0.2cm);
    \draw[gray] (7,-0.6) circle (0.2cm);
    \draw[gray] (7,-1.5) circle (0.2cm);
    \draw[gray] (3.7,-1) circle (0.5cm);
    
    \fill (-1,0.5) circle (2.1pt) node[label, above] {$a_1$};
    \fill (-1,-0.5) circle (2.1pt) node[label, below] {$a_2$};
    \fill (3.7,1.1) circle (2.1pt) node[label, left] {$\textcolor{customblue}{m_*^{-1}}(b_2)$};
    \fill (3.7,-0.8) circle (2.1pt) node[label, below] {$b_2$};
    
    \draw[-{Stealth[scale=1.5]}] (-1,0.5) -- node[left=2pt] {$m$} (-1,-0.5);
    \draw[-{Straight Barb[scale=1.5]}, customblue] (3.7,0.5) -- node[left=2pt, customblue] {$m_*$} (3.7,-0.5);
    \draw[-{Straight Barb[scale=1.5]}] (4.2,1) .. controls (5.5,1) and (5.5,1) .. node[above=2pt] {$C(a_1)$} (6.15,1);
    \draw[-{Straight Barb[scale=1.5]}] (4.2,-1) .. controls (5.1,-1.5) .. node[below=2pt] {$C(a_2)$} (6.15,-1);
    \draw[-{Straight Barb[scale=1.5]}] (4.2,-1) .. controls (5.1,-0.5)  .. node[above=2pt] {$\textcolor{customred}{m_*}(C(a_1))$} (6.1,-0.9);
    \draw[-{Straight Barb[scale=1.5]}, customgreen] (8,0) .. controls (9,1) and (9,-1) .. node[right=2pt, customgreen] {$m_*$} (8,0);
    
    \draw[-{Stealth[scale=1.5]}] (5.1,-0.6) -- (5.1,-1.4);
    
    \draw[-{Stealth[scale=1.5]}] (7,1.3) -- (7,0.8);
    \draw[-{Stealth[scale=1.5]}] (7,0.4) -- (7,-0.4);
    \draw[-{Stealth[scale=1.5]}] (7,-0.8) -- (7,-1.3);
    
    \node[label, anchor=south] at (-1,1.2) {$A$};
    \node[label, anchor=south] at (3.7,1.5) {$B(a_1)$};
    \node[label, anchor=north] at (3.7,-1.5) {$B(a_2)$};
    \node[label, anchor=south] at (7,2) {$\mathtt{Type}$};
    
    \node[label, anchor=west] at (7.2,1.5) {$C(a_1)(\textcolor{customblue}{m_*^{-1}}(b_2))$};
    \node[label, anchor=west] at (7.2,0.6) {$\textcolor{customgreen}{m_*}(C(a_1)(\textcolor{customblue}{m_*^{-1}}(b_2)))$};
    \node[label, anchor=west] at (7.2,-0.6) {$\textcolor{customred}{m_*}(C(a_1))(b_2)$};
    \node[label, anchor=west] at (7.2,-1.5) {$C(a_2)(b_2)$};
    
    \end{tikzpicture}
}

\title{Synthetic Homotopy Theory}   

\author{Yuhang Wei}             
\college{Lady Margaret Hall}  

\degree{MSc Mathematical Sciences}     
\degreedate{Trinity 2024}         

\begin{document}

\baselineskip=18pt plus1pt

\setcounter{secnumdepth}{3}
\setcounter{tocdepth}{3}

\maketitle                  
\chapter*{Abstract}
The goal of this dissertation is to present results from synthetic homotopy theory based on homotopy type theory (HoTT). After an introduction to Martin-Löf's dependent type theory and homotopy type theory, key results include a synthetic construction of the Hopf fibration, a proof of the Blakers--Massey theorem, and a derivation of the Freudenthal suspension theorem, with calculations of some homotopy groups of $n$-spheres.

\textbf{Keywords}: homotopy type theory, homotopy theory, algebraic topology, type theory, constructive mathematics
\chapter*{Acknowledgements}

I would like to thank my dissertation supervisor, Prof. Kobi Kremnitzer, for his invaluable guidance, inspirations and encouragement during this journey. His expertise and support have been instrumental in the completion of this work.

My gratitude also goes to my college supervisor, Dr. Rolf Suabedissen, for providing insightful advice on the structure and direction of this dissertation.

I am profoundly grateful to my parents for their unconditional and unwavering support in my academic pursuits and, indeed, all aspects of my life. Their efforts in their own field of research and teaching have been a constant source of motivation for me.

Finally, special thanks are due to all of my friends who helped me decide my academic path in pure mathematics or supported me in different ways during the writing of this dissertation. An incomplete list of their names follows: Qixuan Fang, Franz Miltz, Samuel Knutsen, Quanwen Chen, Sifei Li, Phemmatad Tansoontorn, Hao Chen and Haoyang Liu. This dissertation would not have been possible without each and every one of them.

\section*{Note}
A slightly abridged version of this dissertation, originally written in Typst, was submitted in fulfillment of the requirements for the degree of MSc Mathematical Sciences at University of Oxford. That version was awarded the Departmental Dissertation Prize by the Mathematical Institute, University of Oxford. The current version has been extended and reformatted in \LaTeX.

\begin{romanpages}          
\tableofcontents            
\listoffigures              
\listoftables
\end{romanpages}            


\chapter{\texorpdfstring{Introduction }{Introduction }}

The goal of this dissertation is to present results in \emph{synthetic
homotopy theory} based on \emph{homotopy type theory} (HoTT, also known
as \emph{univalent foundations}), including the \emph{Hopf fibration}
and \emph{Blakers--Massey theorem}.

Compared to set theory, \emph{type theory} is often considered as a
better foundation for constructive mathematics and mechanised reasoning
\protect{\cite[Section 1.2]{kbh}}. Martin-Löf's \emph{intuitionistic type theory}
(ITT, also known as \emph{dependent type theory}) was first proposed in
the 1970s \protect{\cite{martin}}. In ITT, any mathematical object is seen as inherently
a \emph{term} of some \emph{type}. In particular, ITT follows the
\emph{Curry--Howard correspondence}, or the paradigm of
\emph{propositions-as-types}, where any proposition is regarded as a
type and proving it is equivalent to constructing a term of that type.
Since all terms of types can be seen as computer programmes and be
executed, \emph{proof assistants}, e.g. Coq, Agda and Lean, have been
developed to ensure the correctness of proofs through type-checking
\protect{\cite[Introduction]{brunerie}}.

HoTT is based on the \emph{intensional} version of ITT, where a term of
any \emph{identity type} (i.e., the proposition that two terms of a type
are equal) is not necessarily unique. This seemingly counterintuitive
situation is addressed by interpreting types as topological spaces and
identities as paths between points. The link between homotopy theory and
type theory was first discovered in 2006 independently by Voevodsky and
by Awodey and Warren \protect{\cite{aw}}. In 2009, Voevodsky proposed the
\emph{univalence axiom} (UA), which formalises the idea that `isomorphic
structures have the same properties'. In 2011, the concept of
\emph{higher inductive types} (HIT) emerged \protect{\cite{higher}}, which allows the
use of \emph{path constructors} in addition to \emph{point
constructors}, so that many fundamental topological spaces can be
elegantly constructed. In summary, the components of HoTT are
\[\text{ HoTT } = \text{ ITT } + \text{ UA } + \text{ HIT}.\]

HoTT and \(\infty\)-topoi form a syntax-semantics duality, and all
results proven in HoTT hold in the \(\infty\)-category of homotopy types
\protect{\cite{semantics}}. This lays the foundation for our \emph{synthetic} approach
to homotopy theory, where all constructions are invariant under homotopy
equivalences. Unlike in classical (or \emph{analytical}) homotopy
theory, notions such as spaces, points, continuous functions, and
continuous paths are not based on set-theoretic constructions but are
\emph{axiomatised} as primitive ones, showing the \emph{formal
abstraction} provided by type theory \protect{\cite[Section 1.2]{kbh}}.

Many results of synthetic homotopy theory, including the Hopf fibration,
Freudenthal suspension theorem, and the van Kampen theorem, can be found
in \protect{\cite[Section 8]{hott}}. More recent efforts in this field include
\(\pi_{4}\left( {\mathbb{S}}^{3} \right) \simeq {\mathbb{Z}}/2{\mathbb{Z}}\)
\protect{\cite{brunerie}}, the Blakers--Massey theorem \protect{\cite{blakers}}, the Cayley--Dickson
construction \protect{\cite{cayley}}, the construction of the real projective spaces
\protect{\cite{rpn}}, and the long exact sequence of homotopy \(n\)-groups \protect{\cite{les-group}}.
Additionally, synthetic (co)homology theory has been investigated by
\protect{\cite{syn-coho}}, \protect{\cite{cellular}} and \protect{\cite{syn-ho}}.

This dissertation begins with a brief summary of dependent type theory
in \Cref{sec-type-theory} and an overview of HoTT in \Cref{sec-hott}. \Cref{sec-n-types}
mainly discusses \emph{homotopy \(n\)-types} and \emph{\(n\)-connected
types}. \Cref{sec-homotopy} develops the \emph{fibre sequence}, \emph{Hopf
construction} and \emph{Hopf fibration}. \Cref{sec-blakers} proves the
\emph{Blakers--Massey theorem}. Along the journey, we will calculate the
homotopy groups of \(n\)-spheres in \Cref{tab-homotopy-groups}.

\begin{table}[h]
\centering
\caption{Homotopy groups of $n$-spheres calculated in this dissertation.}
\label{tab-homotopy-groups}
\renewcommand{\arraystretch}{1.4}
\begin{tabular}{ c|>{\centering\arraybackslash}p{1.3cm} >{\centering\arraybackslash}p{1.3cm} >{\centering\arraybackslash}p{1.3cm} c }
\toprule
 & $\mathbb{S}^1$ & $\mathbb{S}^2$ & $\mathbb{S}^3$ & $\mathbb{S}^{n\geq4}$ \\
\midrule
$\pi_1$ & $\mathbb{Z}$ & $0$ & $0$ & \multirow{3}{*}{$\pi_{k < n} (\mathbb{S}^n)\simeq 0$} \\
$\pi_2$ & $0$ & $\mathbb{Z}$ & $0$ & \\
$\pi_3$ & $0$ & $\mathbb{Z}$ & $\mathbb{Z}$ & \\
$\pi_{k\geq4}$ & $0$ & \multicolumn{2}{c }{$\pi_k (\mathbb{S}^2)\simeq \pi_k (\mathbb{S}^3)$} & $\pi_n (\mathbb{S}^n) \simeq  \mathbb{Z}$ \\
\bottomrule
\end{tabular}
\end{table}

Although no result is essentially new, calculations and discussions
original to this dissertation are marked with a red square (\redDiamond). For
instance, \Cref{sec-n-types} contains discussions on path inductions,
\Cref{strong-path-induction} and \Cref{transport-path-induction}, and novel
examples, \Cref{s1-not-1} and \Cref{graph-example}. Most noticeably, \Cref{sec-blakers}
`informalises' a mechanised Blakers--Massey theorem proof \protect{\cite{blakers}} by
integrating details from a Freudenthal suspension theorem proof
\protect{\cite[Section 8.6]{hott}}, before explicitly deriving the Freudenthal
suspension theorem as a corollary. Additionally, visualisation diagrams
are plotted for certain definitions and theorems to aid understanding.

\chapter{Dependent type theory}\label{sec-type-theory}

\textbf{Dependent type theory} is a formal system that organises all
mathematical objects, structures and knowledge by expressing them as
\textbf{types} or \textbf{terms} of types \protect{\cite[p. 1]{rijke}}. A type is
defined by rules specifying its own formation and the construction of
its terms. Types resemble sets; however, a major difference is that
membership in a type is not a predicate but is part of the inherent
nature of its term.

We write \(a:A\) if \(a\) is a term of type \(A\),
where we also say that type \(A\) is \textbf{inhabited} and \(a\) is an
\textbf{inhabitant} of \(A\). All types are considered terms of a
special type, \(\Type\), and we denote \(A:\Type\) if
\(A\) is a type\footnote{We suppress the underlying universe level (see
  \protect{\cite[Section 1.3]{hott}}).}. If \(A:\Type\) and \(a,b:A\), we
write \(a \equiv b\) to indicate that \(a\) and \(b\) are
\textbf{judgementally equal}, meaning that \(a\) and \(b\) are only
symbolically different. When defining a new term \(b:A\) from \(a:A\),
we use the notation \(b: \equiv a\). Judgemental equalities are to be
contrasted with propositional equalities introduced in \Cref{sec-id-type}.

We shall briefly sketch the fundamentals
of dependent type theory, based on \protect{\cite[Chapter I]{rijke}} and
\protect{\cite[Chapter 1]{hott}}. Importantly, these constructions all bear
interpretations in logic and topology in \Cref{interpretation-table}.

\begin{table}[ht]
\centering
\caption{Interpretations of type theory in logic and topology.}
\label{interpretation-table}
\renewcommand{\arraystretch}{1.4}
\begin{tabular}{ >{\centering\arraybackslash}m{3cm} >{\centering\arraybackslash}m{5.2cm} >{\centering\arraybackslash}m{5.2cm} }
\toprule
\textbf{Type theory}& \textbf{Logic} & \textbf{Topology} \\
\midrule
$A : \mathtt{Type}$ & a proposition $A$ & a topological space $A$ \\
$a : A$ & a proof $a$ of the proposition $A$ & a point $a$ in the space $A$ \\
$A \to B$ & if $A$ then $B$ & a continuous map $A \to B$ \\
$P: A \to \mathtt{Type}$ & a predicate $P(x)$ for each $x : A$ & a fibration $P$ over $A$ with fibre $P(x)$ for each $x : A$ \\
$(x: A) \to P(x)$ & for all $x : A$, $P(x)$ & a continuous section of  fibration $P$ \\
$\sum_{x : A} P(x)$ & there exists $x: A$, $P(x)$\tablefootnote{This type carries more information than merely a 'true or false' judgement of the existence, since a term of this type provides a concrete $a : A$ as well as a term $b : P(a)$, i.e., a proof for $P(a)$.} & the total space of  fibration $P$ \\
$A \times B$ & $A$ and $B$ & the product space $A \times B$ \\
$A + B$ & $A$ or $B$\tablefootnote{A term of $A + B$ also carries the information 'which one of $A$ and $B$ is true'.} & the disjoint union $A \sqcup B$ \\
$\mathbb{1}$ & True & the contractible space \\
$\mathbb{0}$ & False & the empty space \\
$p : u =_A v$ & $p$ is a proof of the equality of $u$ and $v$ & $p$ is a continuous path from $u$ to $v$ \\
$\mathtt{refl}_a : a =_A a$ & the reflexivity of equality at $a$ & the constant path at $a$ \\
\bottomrule
\end{tabular}
\end{table}

\section{Function types}

Let \(A:\Type\). A \textbf{dependent type} (or a \textbf{family
of types}) \(P:A \rightarrow \Type\) over \(A\) assigns a type
\(P(x):\Type\) to each term \(x:A\).

A \textbf{dependent function} is a function whose output types may vary
depending on its input. The \textbf{dependent function type}, also known
as \textbf{\(\Pi\)-type}, is specified by:

\begin{enumerate}
\item
  The \emph{formation rule}: Given a type \(A\) and a dependent type
  \(P:A \rightarrow \Type\), the dependent function type
  \((x:A) \rightarrow P(x)\) can be formed.
\item
  The \emph{introduction rule}: To introduce a dependent function of
  type \((x:A) \rightarrow P(x)\), a term \(b(x):P(x)\) is needed for
  each \(x:A\). This dependent function is then denoted by
  \(\lambda x.b(x)\).
\item
  The \emph{elimination rule}: Given a dependent function
  \(f:(x:A) \rightarrow P(x)\) and a term \(a:A\), a term \(f(a):P(a)\)
  is obtained by evaluating \(f\) at \(a\).
\item
  The \emph{computation rule}: Combining the introduction and
  elimination rules, \(\left( \lambda x.b(x) \right)(a): \equiv b(a)\).
\end{enumerate}

In particular, when \(P\) is a constant type family assigning the same
type \(B\) to every \(a:A\), we obtain the \textbf{(non-dependent)
function type} \(A \rightarrow B\).

\section{Inductive types}

A wide range of constructions are based on inductive types. An
\textbf{inductive type} \(A\) is specified by the following:

\begin{enumerate}
\item
  The \emph{constructors} (or \textit{construction rules}) of \(A\) state the ways to construct a term of
  \(A\).
\item
  The \emph{induction principle} states the necessary data to construct
  a dependent function \((x:A) \rightarrow P(x)\) for a given dependent
  type \(P:A \rightarrow \Type\). Particularly, the
  \emph{recursion principle} is the induction principle when \(P\) is a
  constant type.
\item
  The \emph{computation rule }states how the dependent function produced
  by the induction principle acts on the constructors.
\end{enumerate}

\subsection{\texorpdfstring{\(\Sigma\)-types}{\textbackslash Sigma-types}}

Given a type \(A\) and a dependent type
\(P:A \rightarrow \Type\), we can form the
\textbf{\(\Sigma\)-type}, written as \(\sum_{x:A}P(x),\) whose terms are
called \textbf{dependent pairs}.

\begin{enumerate}
\item
  The \emph{construction rule} states that given \(a:A\) and \(b:P(a)\),
  we can construct the dependent pair \((a,b):\sum_{x:A}P(x)\).
\item
  The \emph{induction principle} states that given

  \begin{itemize}
  \item
    a dependent type \(C:\sum_{x:A}P(x) \rightarrow \Type\) and
  \item
    a function
    \(g:(a:A) \rightarrow \left( b:P(a) \right) \rightarrow C(a,b)\)\footnote{Arrows
      \(\rightarrow\) are right-associative.},
  \end{itemize}

  there is a dependent function
  \(f:\left( p:\sum_{x:A}P(x) \right) \rightarrow C(p)\).
\item
  The \emph{computation rule} states that the above \(f\) satisfies
  \(f(a,b): \equiv g(a)(b)\).
\end{enumerate}

The function \(g\) above is the \textbf{curried} form of \(f\) and \(f\)
is the \textbf{uncurried} form of \(g\), and they can be used
interchangeably. Unlike the convention in \protect{\cite{hott}}, we use the notation
that reflects the actual curriedness of a function, i.e., \(g(a)(b)\) is
not written as \(g(a,b)\).

Using the induction principle, we define the first projection function
\(\pr_{1}:\sum_{x:A}P(x) \rightarrow A\) by
\(\lambda a.\lambda b.a:(a:A) \rightarrow \left( b:P(a) \right) \rightarrow A\)
and the second projection function
\(\pr_{2}:\left( p:\sum_{x:A}P(x) \right) \rightarrow P\left( \pr_{1}(p) \right)\)
by
\(\lambda a.\lambda b.b:(a:A) \rightarrow \left( b:P(a) \right) \rightarrow P(a)\).

The case when \(P\) is a constant type \(B\) gives the \textbf{cartesian
product} \(A \times B\). In this case, we write \(\mathtt{fst}\) and
\(\mathtt{snd}\) for the projection maps onto \(A\) and \(B\),
respectively.

\subsection{Other inductive types}

Some other inductive types are specified as follows based
on \protect{\cite[Section 2]{hott-agda}}.

\begin{itemize}
    \item The \textbf{empty type} \(\mathbb{0}\) has no construction rule. The induction principle states that given $ (P : \mathbb{0} \to \mathtt{Type}) $, there is a dependent function $((x : \mathbb{0}) \to P(x))$. In particular, the recursion principle states that for any $A: \Type$, we can define a function $f: \mathbb{0} \to A$, which can be logically interpreted as the principle of explosion, or \textit{ex falso quodlibet}. 
    There is also no computation rule for $\mathbb{0}$. 
    For any type $A$, $\neg A$ is defined as the function type $A \to \mathbb{0}$. 
    \item The \textbf{unit type} $\unt$ has a single constructor $\star : \unt$. The induction principle for $\unt$ states that given a dependent type $P : \unt \to \Type$ and a term $p : P(\star)$, there is a dependent function $f : (x: \unt) \to P(x)$. The computation rule\footnote{Henceforth, we may concisely state the computation rule as a part of the induction principle.} states that such constructed $f$ satisfies $f(\star) :\equiv p$. 
    \item Let $A, B : \Type$. The \textbf{disjoint sum} $A + B$ of $A$ and $B$ has construction rules $\mathtt{inl} : A \to A + B$ and $\mathtt{inr}: B \to A+B$. The induction principle states that given  
    \begin{itemize}
        \item a dependent type $ P : (A+B) \to \mathtt{Type}$,
        \item two dependent functions $f_{\mathtt{inl}} : (a : A) \to P(\mathtt{inl}(a))$ and
        \item $  f_{\mathtt{inr}} : (b: B) \to P(\mathtt{inr}(b))$,
    \end{itemize}
    there is a dependent function $f:((x : A+B) \to P(x)) $ such that $f (\mathtt{inl}(a)) :\equiv f_{\mathtt{inl}} (a)$ and $f (\mathtt{inr}(b)) :\equiv f_{\mathtt{inr}} (b) $.
    \item The \textbf{boolean type} $\mathbb{2}$ is defined as 
    \(\mathbb{2} \equiv \mathbb{1} + \mathbb{1}\), and its two constructors are written as \(\mathtt{True}\) and \(\mathtt{False}\). 
    \item The type of \textbf{natural numbers} $\mathbb{N}$ has construction rules $0: \mathbb{N}$ and $\mathtt{succ} : \mathbb{N} \to \mathbb{N}$. The induction principle states that given
\begin{itemize}
\item a dependent type $P : \mathbb{N} \to \mathtt{Type}$,
\item a term $f_0 : P(0)$, and
\item a dependent function $f_{\mathtt{succ}} :  (n: \mathbb{N}) \to P(n) \to P(\mathtt{succ}(n))$,
\end{itemize}
there is a dependent function $f:((n : \mathbb{N}) \to P(n))$ such that $f(0) : \equiv f_0$ and $f(\mathtt{succ}(n)) : \equiv f_{\mathtt{succ}} (n)(f(n))$.
\item The  type of \textbf{integers} \(\mathbb{Z}\) has three constructors, representing zero,
the positive and negative numbers, and its induction principle and
computation rules are similar to \(\mathbb{N}\) \protect{\cite[Section 1.3]{brunerie}}.
\end{itemize}

\section{Identity types}
\label{sec-id-type}{}

Given a type \(A\) and two terms \(u:A\) and \(v:A\), we can form the
\textbf{identity type} \(u =_{A}^{}v\), whose terms are called
\textbf{(propositional) equalities} or \textbf{paths}. The constructor
for a \emph{family} of identity types \(\left( a =_{A}^{}a \right)\)
indexed by \(a:A\) is given by a dependent function
\[\mathtt{refl}:(a:A) \rightarrow \left( a =_{A}^{}a \right).\]
However, \textbf{Axiom K}, the assumption that
\(p = \mathtt{refl}_{a}\) for any \(p:\left( a =_{A}^{}a \right)\),
does not hold in general.

For a fixed \(a:A\), the induction principle (called the \textbf{path
induction}) for a family of identity types
\(\left( a =_{A}^{}x \right)\) indexed by \(x:A\) states that given

\begin{itemize}
\item
  a dependent type
  \(P:(x:A) \rightarrow \left( a =_{A}^{}x \right) \rightarrow \Type\)
  and
\item
  a term \(d:P(a)\left( \mathtt{refl}_{a} \right)\),
\end{itemize}
there is a dependent function
\[J_{P}(d):(x:A) \rightarrow \left( p:\left( a =_{A}^{}x \right) \right) \rightarrow P(x)(p),\]
such that \(J_{P}(d)(a)\left( \mathtt{refl}_{a} \right): \equiv d\).
The path induction implies that to construct a dependent function for a
path \(p\) from a fixed endpoint \(a\) to a free \(x\), it suffices to
assume that \(x\) is \(a\) and \(p\) is \(\mathtt{refl}_{a}\).
Importantly, one endpoint \(x\) should be free, since otherwise we would
have Axiom K. Intuitively, the path induction reflects the fact that the
space of all paths starting from a fixed point is \emph{contractible}.
We shall revisit this idea in \Cref{sigma-id-contractible}.

The reader is advised to review \Cref{interpretation-table} at this point.

\chapter{Homotopy type theory}\label{sec-hott}

The implications of the identity type are profound, especially when it
interacts with other types. In \Cref{sec-inf}, the identity type equips each
type with an \emph{\(\infty\)-groupoid structure}. In \Cref{sec-dp}, we
explore how equalities can be passed along by non-dependent and
dependent functions. The notions of `equivalence' between functions and
between types are subtle ones treated in \Cref{sec-homo}, and their
relationship with paths in function types or in \(\Type\) is
handled in \Cref{sec-ua} by \emph{function extensionality} and
\emph{univalence axiom}. In \Cref{sec-hit}, we introduce \emph{higher
inductive types} that allow \emph{path constructors}. Together, these
form the basic components of HoTT.

\section{The \texorpdfstring{$\infty$}{infinity}-groupoid structure of types}
\label{sec-inf}

In classical mathematics, equalities are reflexive, symmetric and
transitive. While reflexivity is expressed as the \(\mathtt{refl}\)
constructor for identity types, symmetry and transitivity appear in the
form of path inversion and concatenation.

\begin{lemma}[\protect{\cite[Lemma 2.1.1]{hott}}]
  For every type \(A\) and
\(x,y:A\), there is a function
\(\left( x =_{A}^{}y \right) \rightarrow \left( y =_{A}^{}x \right)\),
denoted \(p \mapsto p^{- 1}\), such that
\(\left( \mathtt{refl}_{x} \right)^{- 1}: \equiv \mathtt{refl}_{x}\),
where \(p^{- 1}\) is called the \textbf{inverse} of \(p\). 
\end{lemma}
\begin{remark} By the paradigm of propositions-as-types, this lemma is
equivalent to the type
\[\left( A:\Type \right) \rightarrow (x:A) \rightarrow (y:A) \rightarrow \left( p:\left( x =_{A}^{}y \right) \right) \rightarrow \left( y =_{A}^{}x \right).\]
The proof shows how we construct a term of this type. \end{remark}

\begin{proof} Fix \(A:\Type\) and \(x:A\). By path induction, it
suffices to assume \(y\) is \(x\) and \(p:\left( x =_{A}^{}y \right)\)
is \(\mathtt{refl}_{x}\). Then it suffices to give a term of
\(x =_{A}^{}x\), which is fulfilled by \(\mathtt{refl}_{x}\). We thus
have
\(\left( \mathtt{refl}_{x} \right)^{- 1}: \equiv \mathtt{refl}_{x}\).
\end{proof}
\begin{lemma}[\protect{\cite[Lemma 2.1.2]{hott}}]
  For every type \(A\) and
  \(x,y,z:A\), there is a function
  \(\left( x =_{A}^{}y \right) \rightarrow \left( y =_{A}^{}z \right) \rightarrow \left( z =_{A}^{}x \right)\),
  denoted \(p \mapsto q \mapsto p{\ \ct\ }\ q\), such that
  \(\mathtt{refl}_{x}{\ \ct\ }\ \mathtt{refl}_{x}: \equiv \mathtt{refl}_{x}\),
  where \(p{\ \ct\ }\ q\) is called the \textbf{concatenation} of \(p\) and
  \(q\). 
\label{path-concat} \end{lemma} \begin{proof}By path
induction twice.\end{proof}

\begin{lemma}[\protect{\cite[Lemma 2.1.4]{hott}}]
  Path concatenation is
  associative, i.e., for any \(x,y,z,w:A\), \(p:x =_{A}^{}y\),
  \(q:y =_{A}^{}z\), and \(r:z =_{A}^{}w\), there is a \(2\)-dimensional
  path
  \(p{\ \ct\ }\ \left( q{\ \ct\ }\ r \right) = \left( p{\ \ct\ }\ q \right){\ \ct\ }\ r\).
\end{lemma}

\begin{proof} By path induction, we assume \(p,q,r\) are all
\(\mathtt{refl}_{x}\). Then
\(\mathtt{refl}_{x}{\ \ct\ }\ \left( \mathtt{refl}_{x}{\ \ct\ }\ \mathtt{refl}_{x} \right) \equiv \mathtt{refl}_{x} \equiv \left( \mathtt{refl}_{x}{\ \ct\ }\ \mathtt{refl}_{x} \right){\ \ct\ }\ \mathtt{refl}_{x}\)
by \Cref{path-concat}, so we give \(\mathtt{refl}_{\mathtt{refl}_{x}}\).
\end{proof}

We similarly can show that path concatenation satisfies inverse and
identity laws \protect{\cite[Lemma 2.1.4]{hott}}. The above are all examples of
\textbf{coherence operations} on paths or higher-dimensional paths. We
omit further discussions and direct to \protect{\cite[Section 2.1]{hott}} or
\protect{\cite[Section 1.4]{brunerie}} for details. Henceforth, we shall omit proofs
which only require path inductions.

We then observe that the identity types can be constructed recursively:
for any type \(A\), we can form \(x =_{A}^{}y\),
\(p =_{x =_{A}^{}y}^{}q\), and \(r =_{p =_{x =_{A}^{}y}^{}q}^{}s\), etc.
This observation, combined with the next definition, \emph{pointed
types}, gives rise to the notion of \emph{loop spaces}.

\begin{definition} The type of \textbf{pointed types} is defined as
\(\sum_{A:\Type}A\). Alternatively, a pointed type is a
dependent pair \(\left( A, \star_{A} \right)\) where \(A:\Type\)
and \(\star_{A}:A\), and \(\star_{A}\) is called the \textbf{basepoint}
of this pointed type. We often leave the basepoint implicit. \end{definition}

\begin{example} Let \(A:\Type\), then define a pointed type
\(A_{+}: \equiv \left( A + \mathbb{1},\mathtt{inr}( \star ) \right)\).
This is called \textbf{adjoining a disjoint basepoint}. \end{example}

\begin{definition} The \textbf{loop space} of a pointed type
\(\left( A, \star_{A} \right)\) is defined as the pointed type
\[\Omega\left( A, \star_{A} \right): \equiv \left( \left( \star_{A} =_{A}^{} \star_{A} \right),\mathtt{refl}_{\star_{A}} \right).\]
Then for \(n:{\mathbb{N}}\), define the \textbf{\(n\)-th loop space} of
\(\left( A, \star_{A} \right)\) by \(\Omega^{0}A: \equiv A\) and
\(\Omega^{n + 1}A: \equiv \Omega\left( \Omega^{n}A \right)\) (basepoints
implicit). \end{definition}

In summary, coherence operations and loop spaces demonstrate that the
identity type equips each pointed type with a \emph{(weak)
\(\infty\)-groupoid} structure (see \protect{\cite[Appendix A]{brunerie}} for a
formal definition).

\section{Dependent paths}\label{sec-dp}

In classical mathematics, for any function between sets
\(f:A \rightarrow B\) and \(x,y \in A\), \(x = y\) implies
\(f(x) = f(y)\). The analogue in HoTT is expressed as a path
\(f(x) = f(y)\) induced by the path \(x = y\), demonstrating the
\emph{proof relevance} feature of type theory.

\begin{lemma}[\protect{\cite[Lemma 2.2.1]{hott}}]
  Let \(A,B:\Type\),
 \(f:A \rightarrow B\), \(x,y:A\) and \(p:x =_{A}^{}y\). Then there is a
 path \(\mathtt{ap}_{f}(p):f(x) =_{B}^{}f(y)\), such that
 \(\mathtt{ap}_{f}\left( \mathtt{refl}_{x} \right): \equiv \mathtt{refl}_{f(x)}\).
\label{ap} \end{lemma} However, for a dependent function
\(f:(x:A) \rightarrow P(x)\), there is not necessarily a path from
\(f(x)\) to \(f(y)\), which generally belong to different types \(P(x)\)
and \(P(y)\), even when there is \(p:x =_{A}^{}y\). We have to
`transport' \(f(x)\) from \(P(x)\) to \(P(y)\) `along' \(p\) first,
shown by the following two lemmas, visualised in \Cref{fig-trans}.

\begin{lemma}[\protect{\cite[Lemma 2.3.1]{hott}}]
  Let \(A:\Type\), \(P:A \rightarrow \Type\), \(x,y:A\) and \(p:x =_{A}^{}y\). Then
  there is a function
  \(p_{\ast} \equiv \mathtt{transport}^{P}(p):P(x) \rightarrow P(y)\),
  such that
  \(\left( \mathtt{refl}_{x} \right)_{\ast}: \equiv \operatorname{id}\limits_{P(x)}\).
  \label{transport} \end{lemma} 
   
   \begin{lemma}[\protect{\cite[Lemma 2.3.4]{hott}}]
 Let \(A:\Type\),
\(P:A \rightarrow \Type\), \(f:(a:A) \rightarrow P(a)\),
\(x,y:A\) and \(p:x =_{A}^{}y\). Then there is a \textbf{dependent path}
\(\mathtt{apd}_{f}(p):f(x) =_{p}^{P}f(y)\) such that
\(\mathtt{apd}_{f}\left( \mathtt{refl}_{x} \right): \equiv \mathtt{refl}_{f(x)}\),
where we define the type
\[\left( u =_{p}^{P}v \right): \equiv \left( \mathtt{transport}^{P}(p)(u) =_{P(y)}^{}v \right)\]
for any \(u:P(x)\) and \(v:P(y)\). \label{apd-path} \end{lemma}


\begin{figure}[H]
\centering
\figTrans
\caption{Diagram for \Cref{transport} and \Cref{apd-path}.}
\caption*{\small In this and following diagrams, a type is represented as a grey solid circle or ellipse, a term as a black solid dot, a path as a \tikz[baseline=-0.5ex]\draw[-{Stealth[scale=1.5]}, thick] (0,0) -- (0.5,0);  arrow, and a non-dependent function as  a \tikz[baseline=-0.5ex]\draw[-{Straight Barb[scale=1.5]}, thick] (0,0) -- (0.5,0);   arrow.}
\label{fig-trans}

\end{figure}

The next lemma relates the dependent and non-dependent cases, visualised in \Cref{fig-const}.

\begin{lemma}[\protect{\cite[Lemma 2.3.5, Lemma 2.3.8]{hott}}]
 Let
\(A,B:\Type\), \(x,y:A\), and \(p:x =_{A}^{}y\). Let
\(P:A \rightarrow \Type\) be a constant family of types given by
\(\lambda a.B\). Then for any \(b:B\), there is a path
\[\mathtt{tr\text{-}const}_{p}^{B}(b):\mathtt{transport}^{P}(p)(b) =_{B}^{}b.\]
Let \(f:A \rightarrow B\), then
\(\mathtt{apd}_{f}(p) = \mathtt{tr\text{-}const}_{p}^{B}\left( f(x) \right)\ct\mathtt{ap}_{f}(p)\).

\label{transport-constant} \end{lemma}

\begin{figure}[H]
\centering
\figTwo
\caption{Diagram for \Cref{transport-constant}.}
\label{fig-const}
\end{figure}

The next result shows the functoriality of \texttt{transport}. 
  
  \begin{proposition}[\protect{\cite[Lemma 2.3.9]{hott}}]
    Let \(P:A \rightarrow \Type\), \(p:x =_{A}^{}y\)
    and \(q:y =_{A}^{}z\). Then
    \(\mathtt{transport}^{P}(p\ct q) = \mathtt{transport}^{P}(q)\circ\mathtt{transport}^{P}(p).\)
  \label{transport-functorality} \end{proposition}

It is helpful to characterise \texttt{transport} when
the dependent type \(P:A \rightarrow \Type\) is of particular
forms, as in the following lemmas, visualised in \Cref{fig-1}, \Cref{fig-2} and
\Cref{fig-3}.

\begin{lemma} Let \(A,B:\Type\), \(f,g:A \rightarrow B\),
\(x,y:A\), \(p:x =_{A}^{}y\) and \(u:f(x) =_{B}^{}g(x)\). Let
\(P: \equiv \lambda a.\left( f(a) =_{B}^{}g(a) \right):A \rightarrow \Type\).
Then
\[\mathtt{transport}^{P}(p)(u) =_{f(y) =_{B}^{}g(y)}^{}\left( \mathtt{ap}_{f}(p) \right)^{- 1}\ct u\ct\mathtt{ap}_{g}(p).\]
\label{identity-path-1} \end{lemma}

\begin{figure}[H]
\centering
\figOne
\caption{Diagram for \Cref{identity-path-1}.}
\caption*{\small The two dashed ellipses represent identity types \(f(x) = g(x)\) and \(f(y) = g(y)\).}
\label{fig-1}
\end{figure}

\begin{lemma}[\protect{\cite[(2.9.4)]{hott}}]
 Let \(X:\Type\) and
\(A,B:X \rightarrow \Type\). Let
\(P: \equiv \lambda x.A(x) \rightarrow B(x):X \rightarrow \Type\).
Let \(x_{1},x_{2}:X\), \(p:x_{1} =_{X}^{}x_{2}\). Let
\(f:P\left( x_{1} \right) \equiv \left( A\left( x_{1} \right) \rightarrow B\left( x_{1} \right) \right)\).
Then
\[\mathtt{transport}^{P}(p)(f) =_{P\left( x_{2} \right)}^{}\mathtt{transport}^{B}(p) \circ f \circ \mathtt{transport}^{A}\left( p^{- 1} \right).\]

\label{function-identity} \end{lemma}

\begin{figure}[H]
\centering
\figTransFun
\caption{Diagram for \Cref{function-identity}.}
\caption*{\small The two dashed ellipses represent the function types \(A\left( x_{1} \right) \rightarrow B\left( x_{1} \right)\) and \(A\left( x_{2} \right) \rightarrow B\left( x_{2} \right)\).}
\label{fig-2}
\end{figure}

\begin{lemma}[\protect{\cite[Lemma 2.11.2]{hott}}]
 Let \(A:\Type\),
\(a,x_{1},x_{2}:A\) and \(p:x_{1} =_{A}^{}x_{2}\). Then 
\begin{align*}
    \mathtt{transport}^{x \mapsto (a=x)}(p)(q_1) &= q_1 \,\ct\, p \quad &\text{for } q_1 : a = x_1, \\
    \mathtt{transport}^{x \mapsto (x=a)}(p)(q_2) &= p^{-1} \,\ct\, q_2 \quad &\text{for } q_2 : x_1 = a, \\
    \mathtt{transport}^{x \mapsto (x=x)}(p)(q_3) &= p^{-1} \,\ct\, q_3 \,\ct\, p \quad &\text{for } q_3 : x_1 = x_1.
\end{align*}
\label{transport-path} \end{lemma}

\begin{figure}[H]
    \centering
    \figTransPath
    \caption{Diagram for the first case of \Cref{transport-path}.}
    \caption*{\small The two
 dashed ellipses represent identity types \(a = x_{1}\) and
 \(a = x_{2}\).}
    \label{fig-3}
\end{figure}

\section{Homotopies and equivalences}\label{sec-homo}

The next problem is how to construct equalities between functions and
between types. In this section, we shall present two constructions,
\emph{function homotopy} and \emph{type equivalence}, based on the
identity types.

\begin{definition} Let \(A:\Type\),
\(B:A \rightarrow \Type\) and \(f,g:(a:A) \rightarrow B(a)\). A
\textbf{(function) homotopy} from \(f\) to \(g\) is a dependent function
of type
\[(f \sim g): \equiv (a:A) \rightarrow \left( f(a) =_{B(a)}^{}g(a) \right),\]
i.e., a family of pointwise equalities. \end{definition}

\begin{definition} Let \(A,B:\Type\), \(f:A \rightarrow B\). A
\textbf{quasi-inverse} of \(f\) is a triple \((g,\alpha,\beta)\), where
\(g:B \rightarrow A\),
\(\alpha:f\circ g \sim \operatorname{id}\limits_{B}\) and
\(\beta:g\circ f \sim \operatorname{id}\limits_{A}\). \end{definition}

\begin{proposition}[\protect{\cite[Section 2.4]{hott}}]
 Let \(A,B:\Type\)
and \(f:A \rightarrow B\). Define type
\[\mathtt{isequiv}(f): \equiv \left( \sum_{g:B \rightarrow A}\left( g\circ f \sim \operatorname{id}\limits_{A} \right) \right) \times \left( \sum_{h:B \rightarrow A}\left( f\circ h \sim \operatorname{id}\limits_{B} \right) \right).\]
We say `\(f\) is an \textbf{equivalence}' if \(\mathtt{isequiv}(f)\)
is inhabited. Then

\begin{enumerate}
\item
  \(f\) is an equivalence if and only if \(f\) has a quasi-inverse;
\item
  Given any \(p,q:\mathtt{isequiv}(f)\), there is a path
  \(p =_{\mathtt{isequiv}(f)}^{}q\).
\end{enumerate}

\end{proposition}

In the language of \Cref{sec-n-type}, (2) states that
\(\mathtt{isequiv}(f)\) is a \emph{mere proposition}. In practice, to
prove that a function \(f\) is an equivalence, by (1), we only have to
find a quasi-equivalence of \(f\). An example follows.

\begin{lemma}[\protect{\cite[Example 2.4.9]{hott}}]
 Let \(A:\Type\),
\(P:A \rightarrow \Type\), \(x,y:A\) and \(p:x =_{A}^{}y\). Then
\(\mathtt{transport}^{P}(p):P(x) \rightarrow P(y)\) is an equivalence.

\label{transport-isequiv} \end{lemma} \begin{proof}By
\Cref{transport-functorality},
\(\mathtt{transport}^{P}\left( p^{- 1} \right)\) forms a quasi-inverse
with \(\mathtt{transport}^{P}(p)\).\end{proof}

\begin{definition} Let \(A,B:\Type\). Then we define the type of
\textbf{equivalences} between \(A\) and \(B\) as
\[(A \simeq B): \equiv \sum_{f:A \rightarrow B}\mathtt{isequiv}(f).\]
\end{definition}

We now characterise the identity type in certain types using the
language of equivalence:
\begin{proposition}[\protect{\cite[Proposition 1.6.8]{brunerie}}]
  Let \(A:\Type\),
\(B:A \rightarrow \Type\) and two terms
\((a,b),(a^{\prime},b^{\prime}):\sum_{x:A}B(x)\), then there is an equivalence
of types
\[\left( (a,b) =_{\sum_{x:A}B(x)}^{}(a^{\prime},b^{\prime}) \right) \simeq \left( \sum_{p:a =_{A}^{}a^{\prime}}b =_{p}^{B}b^{\prime} \right).\]
The function from right to left is written as
\(\mathtt{pair}\hspace{0pt}^{=}\), i.e., given \(p:a =_{A}^{}a^{\prime}\)
and \(\ell:b =_{p}^{B}b^{\prime}\), there is a path
\(\mathtt{pair}\hspace{0pt}^{=}(p,\ell):(a,b) = (a^{\prime},b^{\prime})\).
\label{sigma-path} \end{proposition}

\section{Univalence axiom and function extensionality}\label{sec-ua}

Given \(A,B:\Type\), it is natural to wonder what is the
relationship between type equality \(A =_{\Type}^{}B\) and type
equivalence \(A \simeq B\). It turns out that we can derive the latter
from the former:

\begin{proposition}[\protect{\cite[Lemma 2.10.1]{hott}}]
 Let
\(A,B:\Type\). Then there is a function
\(\mathtt{idtoeqv}:\left( A =_{\Type}^{}B \right) \rightarrow (A \simeq B)\).

\end{proposition} \begin{proof} Let \(p:A =_{\Type}^{}B\). The identity function
\(\operatorname{id}\limits_{\Type} \equiv \lambda X.X:\Type \rightarrow \Type\)
can be regarded as a dependent type, so by \Cref{transport}, we have a
function
\(\mathtt{idtoeqv}(p): \equiv \mathtt{transport}^{\operatorname{id}\limits_{\Type}}(p):A \rightarrow B\),
which is an equivalence by \Cref{transport-isequiv}. \end{proof}

Desirable as it is, we however cannot go the other way around with what
we currently have. Our wish is instead formulated as an axiom:
\begin{axiom}[Univalence] Let \(A,B:\Type\). Then
\(\mathtt{idtoeqv}:\left( A =_{\Type}^{}B \right) \rightarrow (A \simeq B)\)
is an equivalence. \end{axiom}

\begin{remark} The \emph{axiom} means that we assume the type
\(\mathtt{isequiv}(\mathtt{idtoeqv})\) is inhabited without
constructing its term. \end{remark} In particular, this implies that there is a
function
\[\mathtt{ua}:(A \simeq B) \rightarrow \left( A =_{\Type}^{}B \right), \]
which forms a quasi-inverse of \(\mathtt{idtoeqv}\).
By abuse of notation, we write
\(\mathtt{ua}(f):\left( A =_{\Type}^{}B \right)\) if
\(f:A \rightarrow B\) is an equivalence. In practice, to find a path
between two types \(A\) and \(B\), it suffices to give a function
\(f:A \rightarrow B\) and show that \(f\) is an equivalence.

A similar situation arises with function equality and function homotopy.

\begin{proposition} Let \(A:\Type\),
\(B:A \rightarrow \Type\) and \(f,g:(a:A) \rightarrow B(a)\).
Then there is a function
\(\mathtt{happly}:\left( f =_{(a:A) \rightarrow B(a)}^{}g \right) \rightarrow (f \sim g)\).
\end{proposition} \begin{proof} Let
\(p:\left( f =_{(a:A) \rightarrow B(a)}^{}g \right)\). For any \(a:A\),
the evaluation function
\(\lambda h.h(a):\left( (a:A) \rightarrow B(a) \right) \rightarrow B(a)\)
is a non-dependent one, so we can apply \Cref{ap} and get
\(\mathtt{happly}(p)(a): \equiv \mathtt{ap}_{\lambda h.h(a)}(p):\left( f(a) =_{B(a)}^{}g(a) \right).\)
\end{proof} The next theorem can be proved with univalence \protect{\cite[Section 4.9]{hott}}. 
\begin{theorem}[Function extensionality]
  The function
  \(\mathtt{happly}\) is an equivalence.
\end{theorem}
   This implies that there is
a function
\[\mathtt{funext}:\left( (a:A) \rightarrow \left( f(a) =_{B(a)}^{}g(a) \right) \right) \rightarrow \left( f =_{(a:A) \rightarrow B(a)}^{}g \right),\]
which forms a quasi-inverse of \(\mathtt{happly}\). In practice, to
give an equality between two functions \(f\) and \(g\), it suffices to
give a family of pointwise equalities.

\section{Higher inductive types}\label{sec-hit}

In classical dependent type theory, any constructor for an inductive
type \(A\) should produce a \emph{term} of \(A\). This is not the case
for \textbf{higher inductive types} (\textbf{HIT}), where \textbf{path
constructors} are allowed which produce \emph{paths}, or `customised
equalities', between points of \(A\). We now introduce two HITs, the
circle and pushouts.

\subsection{Circle}\label{sec-circle}

The circle \({\mathbb{S}}^{1}\) is the HIT \emph{generated} by a
\emph{(point) constructor} \(\mathtt{base}:{\mathbb{S}}^{1}\) and a
\emph{path constructor}
\(\mathtt{loop}:\mathtt{base} =_{{\mathbb{S}}^{1}}^{}\mathtt{base}\).
The wording here is `generated', in the sense that by coherence
operations, from the constructor \(\mathtt{loop}\) we automatically
obtain paths such as \(\mathtt{loop}^{- 1}\) and
\(\mathtt{loop}\ct\mathtt{loop}\ct\mathtt{loop}\).
Note that \({\mathbb{S}}^{1}\) is non-trivial, in the sense that
\(\mathtt{loop} \neq \mathtt{refl}_{\mathtt{base}}\) \footnote{Note that this means the type \((\mathtt{loop} = \mathtt{refl}_\mathtt{base}) \to \mathbb{0}\) is inhabited.}\protect{\cite[Lemma 6.4.1]{hott}}.

The induction principle states that given

\begin{itemize}
\item
  a dependent type \(P:{\mathbb{S}}^{1} \rightarrow \Type\),
\item
  a term \(b:P\left( \mathtt{base} \right)\) and
\item
  a dependent path \(\ell:b =_{\mathtt{loop}}^{P}b\),
\end{itemize}
there is a dependent function
\(f:\left( x:{\mathbb{S}}^{1} \right) \rightarrow P(x)\) such that
\(f\left( \mathtt{base} \right): \equiv b\) and
\(\mathtt{apd}_{f}\left( \mathtt{loop} \right) := \ell\) \footnote{Here
  the computation rule for the path is a propositional equality instead
  of a judgemental one; see \protect{\cite[Section 6.2]{hott}}.}.

\begin{remark}[\protect{\cite[Remark 6.2.4]{hott}}]
 When informally applying
circle induction on \(x:{\mathbb{S}}^{1}\) given
\(P:{\mathbb{S}}^{1} \rightarrow \Type\), we may use expressions
`when \(x\) is \(\mathtt{base}\)' and `when \(x\) varies along
\(\mathtt{loop}\)' to introduce the constructions of \(b\) and
\(\ell\) as above. 
\end{remark}

The recursion principle states that given
\begin{itemize}
\item
  \(B:\Type\),
\item
  \(b:B\) and
\item
  \(p:b =_{B}^{}b\),
\end{itemize}
there is a function \(f:{\mathbb{S}}^{1} \rightarrow B\) such that
\(f\left( \mathtt{base} \right): \equiv b\) and
\(\mathtt{ap}_{f}\left( \mathtt{loop} \right) := p\) \protect{\cite[Lemma 6.2.5]{hott}}.

\subsection{Pushouts}\label{sec-pushout}

Let \(A,B,C:\Type\), \(f:C \rightarrow A\) and
\(g:C \rightarrow B\). Then we get a diagram
\(A\overset{f}{\leftarrow}C\overset{g}{\rightarrow}B\), which we call a
\textbf{span}. The \textbf{pushout} \(A \sqcup^{C}B\) of this span is
the HIT generated by two point constructors
\(\mathtt{inl}:A \rightarrow A \sqcup^{C}B\) and
\(\mathtt{inr}:B \rightarrow A \sqcup^{C}B\) and a path constructor
\[\mathtt{glue}:(c:C) \rightarrow \left( \mathtt{inl}(f(c)) =_{A \sqcup^{C}B}^{}\mathtt{inr}(g(c)) \right),\]
as visualised by \Cref{fig-pushout}.

The induction principle states that given

\begin{itemize}
\item
  a dependent type \(P:A \sqcup^{C}B \rightarrow \Type\),
\item
  two dependent functions
  \(h_{\mathtt{inl}}:(a:A) \rightarrow P\left( \mathtt{inl}(a) \right)\)
  and
\item
  \(h_{\mathtt{inr}}:(b:B) \rightarrow P\left( \mathtt{inr}(b) \right)\),
  and
\item
  a family of dependent paths
  \[h_{\mathtt{glue}}:(c:C) \rightarrow \left( h_{\mathtt{inl}}\left( f(c) \right) =_{\mathtt{glue}(c)}^{P}h_{\mathtt{inr}}\left( g(c) \right) \right),\]
\end{itemize}
there is a dependent function
\(h:\left( x:A \sqcup^{C}B \right) \rightarrow P(x)\) such that
\(h\left( \mathtt{inl}(a) \right): \equiv h_{\mathtt{inl}}(a)\),
\(h\left( \mathtt{inr}(b) \right): \equiv h_{\mathtt{inr}}(b)\) and
\(\mathtt{apd}_{h}\left( \mathtt{glue}(c) \right) := h_{\mathtt{glue}}(c)\).

\begin{figure}
    \centering
    \figPoOne
    \caption{Diagram for the pushout \(A \sqcup^{C}B\).}
    \caption*{The two dashed circles represent
 the `image' of \(A\) and \(B\) under \(\mathtt{inl}\) and
\(\mathtt{inr}\), respectively.}
    \label{fig-pushout}
\end{figure}

By specialising the span, we can obtain many interesting constructions
from the pushout.

\begin{enumerate}
\item
  Let \(A:\Type\). The \textbf{suspension} \(\Sigma A\) of \(A\)
  is the pushout of
  \[\mathbb{1}_{N} \leftarrow A \rightarrow \mathbb{1}_{S}.\]
We denote \(\mathtt{north}: \equiv \mathtt{inl}( \star_{N})\),
\(\mathtt{south}: \equiv \mathtt{inr}( \star_{S})\), and for any
\(a:A\),
\(\mathtt{merid}(a): \equiv \mathtt{glue}(a):\mathtt{north} =_{\Sigma A}^{}\mathtt{south}\).
\(\Sigma A\) can be seen as two points ('north and south poles')
connected by a family of paths ('meridians') indexed by the terms of
\(A\).
\item
  Let \(A,B:\Type\). The \textbf{join} \(A \ast B\) of \(A\) and
  \(B\) is the pushout of
  \[A\overset{\mathtt{fst}}{\leftarrow}A \times B\overset{\mathtt{snd}}{\rightarrow}B.\]

\item
  Let \(\left( A, \star_{A} \right)\) and
  \(\left( B, \star_{B} \right)\) be two pointed types. The
  \textbf{wedge sum} \(A \vee B\) of \(A\) and \(B\) is the pushout of
  \[A\overset{\star_{A}}{\leftarrow}\mathbb{1}\overset{\star_{B}}{\rightarrow}B.\]
\end{enumerate}

With suspension, we define the \textbf{\(n\)-sphere}
\({\mathbb{S}}^{n}\) by induction on \(n \geq 0\):
\({\mathbb{S}}^{0}: \equiv \mathbb{2}\) and
\({\mathbb{S}}^{n + 1}: \equiv \Sigma{\mathbb{S}}^{n}\). Note that
\({\mathbb{S}}^{1} \simeq \Sigma\mathbb{2}\) \protect{\cite[Lemma 6.5.1]{hott}},
where \({\mathbb{S}}^{1}\) is in the sense of \Cref{sec-circle}.

For any type \(X\), we define its \textbf{pointed suspension} as
\(\left( \Sigma X,\mathtt{north} \right).\) This turns
\({\mathbb{S}}^{n}\) into pointed types for \(n \geq 1\). For \(n = 0\),
we view \({\mathbb{S}}^{0} \equiv \mathbb{2}\) as \(\mathbb{1}_{+}\). We
write the basepoint of \({\mathbb{S}}^{n}\) as \(\mathtt{base}\).

\chapter{\texorpdfstring{Homotopy \(n\)-types and \(n\)-connected
types}{Homotopy n-types and n-connected types}}\label{sec-n-types}

\Cref{sec-contr} discusses \emph{contractible types}, `trivial types' that
are equivalent to \(\mathbb{1}\). \Cref{sec-n-type} generalises contractible
types to (\emph{homotopy) \(n\)-types}: types with trivial `homotopy
information' in any dimension greater than \(n\) (as in
\Cref{omega-contractible}). \Cref{sec-trun} discusses how to truncate any given
type into an \(n\)-type. Then our tools at hand will allow us to define
and calculate the homotopy groups of the circle in \Cref{sec-algebra}.
Finally, \Cref{sec-connect} focuses on the `dual' of \(n\)-types,
\emph{\(n\)-connected types}, which have trivial `homotopy information'
in any dimension less than or equal to \(n\).

\section{Contractible types}\label{sec-contr}

\begin{definition} For \(A:\Type\), we define the predicate
\[\mathtt{is\text{-}contr}(A): \equiv \sum_{a:A}\left( (x:A) \rightarrow \left( a =_{A}^{}x \right) \right)\]
which reads `\(A\) is \textbf{contractible}', and we say \(a\) is the
\textbf{centre of contraction}. \end{definition}

Clearly, \(\mathbb{1}\) is a contractible type. In fact, \(A\) is
contractible if and only if \(A \simeq \mathbb{1}\) \protect{\cite[Lemma 3.11.3]{hott}}. Informally, contractibility means having trivial `homotopy
information' in all dimensions (observe that
\(\Omega^{n}\left( \mathbb{1} \right) \simeq \mathbb{1}\) for all
\(n:{\mathbb{N}}\)). We first give a non-example.

\begin{example}\redDiamond \({\mathbb{S}}^{1}\) is not
contractible. Otherwise, taking 	\texttt{base} as the
centre of contraction, we have to prove
\(\left( x:{\mathbb{S}}^{1} \right) \rightarrow \left( \mathtt{base} = x \right)\).
By circle induction, when \(x\) is 	\texttt{base}, we
need \(p:\mathtt{base} = \mathtt{base}\). When \(x\) `varies along'
	\texttt{loop}, we need
\(\mathtt{loop}_{\ast}(p) = p\). By \Cref{transport-path},
\(\mathtt{loop}_{\ast}(p) = p\ct\mathtt{loop}\), so by
coherence operations, we need
\(\mathtt{loop} = \mathtt{refl}_{\mathtt{base}}\), but this would give us absurdity by the non-triviality of \({\mathbb{S}}^{1}\) \protect{\cite[Lemma 6.4.1]{hott}}. \label{s1-not-1} \end{example}

As mentioned earlier, the space of all paths from a fixed point is
contractible:

\begin{proposition} Let \(X:\Type\) and \(a:X\). Then the type
\(\sum_{x:X}\left( a =_{X}^{}x \right)\) is contractible. \label{sigma-id-contractible} \end{proposition} \begin{proof} The centre of
contraction is given by \(\left( a,\mathtt{refl}_{a} \right)\). For
any \((x,p):\sum_{x:X}\left( a =_{X}^{}x \right)\), we have \(p:a = x\)
and \(\ell:p_{\ast}\left( \mathtt{refl}_{a} \right) = p\) by
\Cref{transport-path}. Then we have
\(\mathtt{pair}\hspace{0pt}^{=}(p,\ell):\left( a,\mathtt{refl}_{a} \right) = (x,p)\)
by \Cref{sigma-path}. \end{proof} \redDiamond This lends us some further
insights into path induction, as in the two following corollaries.
\begin{corollary}[Path induction, strengthened]
  Let
\(A:\Type\), \(a:A\), and
\(P:(x:A) \rightarrow \left( a =_{A}^{}x \right) \rightarrow \Type\).
Then for any \(y:A\), \(q:a =_{A}^{}y\) and \(d:P(y)(q)\), there is a
function
\(f:(x:A) \rightarrow \left( p:\left( a =_{A}^{}x \right) \right) \rightarrow P(x)(p)\).
\label{strong-path-induction} \end{corollary} \begin{proof} Let
\(\hat{P}:\sum_{x:A}\left( a =_{A}^{}x \right) \rightarrow \Type\)
be the uncurried form of \(P\). Then by the induction principle of
\(\mathbb{1}\), it suffices to take any
\((y,q):\sum_{x:A}\left( a =_{A}^{}x \right)\) and
\(d:\hat{P}(y,q) \equiv P(y)(q)\) to obtain
\(\hat{f}:\left( (x,p):\sum_{x:A}\left( a =_{A}^{}x \right) \right) \rightarrow \hat{P}(x,p)\).
The desired \(f\) is the curried form of \(\hat{f}\). \end{proof} \begin{remark} In
the usual path induction, \(y\) and \(q\) should be \(a\) and
\(\mathtt{refl}_{a}\) respectively. \end{remark} We may also use
	\texttt{transport} to express functions defined using
path induction: \begin{corollary} Let \(A:\Type\), \(a,x:A\),
\(p:\left( a =_{A}^{}x \right)\),
\(P:(x:A) \rightarrow \left( a =_{A}^{}x \right) \rightarrow \Type\)
and \(d:P(a)\left( \mathtt{refl}_{a} \right)\). Then
\[J_{P}(d)(x)(p) =_{P(x)(p)}^{}\mathtt{transport}^{\hat{P}}\left( \mathtt{pair}\hspace{0pt}^{=}(p,\ell) \right)(d),\]
where \(\ell:p_{\ast}\left( \mathtt{refl}_{a} \right) = p\) by
\Cref{transport-path}, \(J\) is defined in \Cref{sec-id-type} and
\(\hat{P}:\sum_{x:A}\left( a =_{A}^{}x \right) \rightarrow \Type\)
is the uncurried form of \(P\). \label{transport-path-induction} \end{corollary} \begin{proof}By path
induction, assume \(x\) is \(a\) and \(p\) is \(\mathtt{refl}_{a}\),
then both sides are \(d\).\end{proof}

\section{\texorpdfstring{Homotopy
\(n\)-types}{Homotopy n-types}}\label{sec-n-type}

\begin{definition} We recursively define the predicate
\[\mathtt{is\text{-}}n\mathtt{\text{-}Type}(A): \equiv \begin{cases}
\mathtt{is\text{-}contr}(A)\quad & \text{if  }n = - 2, \\
(x:A) \rightarrow (y:A) \rightarrow \mathtt{is\text{-}}(n - 1)\mathtt{\text{-}Type}(x =_{A}^{}y)\quad & \text{if  }n \geq - 1,
\end{cases}\] for any \(n \geq - 2\), which reads `\(A\) is an
\textbf{\(n\)-type}'. We also define
\[n\mathtt{\text{-}Type}: \equiv \sum_{X:\Type}\mathtt{is\text{-}}n\mathtt{\text{-}Type}(X)\]
and in particular, we write
\(\mathtt{Prop}: \equiv ( - 1)\mathtt{\text{-}Type}\). \end{definition}

The type \(\mathtt{is\text{-}}n\mathtt{\text{-}Type}(A)\) is a
\(( - 1)\)-type \protect{\cite[Theorem 7.1.10]{hott}}, and the type
\(n\mathtt{\text{-}Type}\) is an \((n + 1)\)-type \protect{\cite[Theorem 7.1.11]{hott}}.

A \(( - 2)\)-type is by definition a contractible type; we also define a
\textbf{mere proposition} as a \(( - 1)\)-type and a \textbf{set} as a
\(0\)-type. They are characterised as follows:

\begin{itemize}
\item
  \(A:\Type\) is contractible, if and only if \(A\) is a pointed
  mere proposition, if and only if \(A \simeq \mathbb{1}\) \protect{\cite[Lemma 3.11.3]{hott}}.
\item
  \(A:\Type\) is a mere proposition if and only if for any
  \(x,y:A\), we have \(x =_{A}^{}y\) \protect{\cite[Definition 3.3.1 and Lemma 3.11.10]{hott}}. This aligns with the `proof-irrelevant' propositions in
  classical logic. For example, \(\mathbb{0}\) is a mere proposition
  \protect{\cite[Proposition 1.10.5]{brunerie}}.
\item
  \(A:\Type\) is a set, if and only if for any \(x,y:A\) and
  \(p,q:x =_{A}^{}y\) we have \(p =_{x =_{A}^{}y}^{}q\), if and only if
  \(A\) satisfies Axiom K, i.e., \(a =_{A}^{}a\) is contractible for any
  \(a:A\) \protect{\cite[Theorem 7.2.1]{hott}}. This aligns with the
  `equality-irrelevant' sets in classical mathematics and allows us to talk about set-theoretic notions such as subsets. The types
  \(\mathbb{2}\), \(\mathbb{N}\) and \(\mathbb{Z}\) are sets
  \protect{\cite[Proposition 1.10.7]{brunerie}}.
\end{itemize}

As a generalisation of Axiom K, we can characterise \(n\)-types in terms
of loop spaces: 
\begin{theorem}[\protect{\cite[Theorem 7.2.9]{hott}}]
  For any \(n \geq - 1\), a type \(A\) is an \(n\)-type if and only if
  \(\Omega^{n + 1}(A,a)\) is contractible for all \(a:A\).
\label{omega-contractible} \end{theorem}

Noticeably, \(n\)-types have the cumulativity property:

\begin{proposition}[\protect{\cite[Theorem 7.1.7]{hott}}]
 Let
\(m \geq n \geq - 2\). An \(n\)-type is also an \(m\)-type. 
\end{proposition}

\section{\texorpdfstring{\(n\)-truncations}{n-truncations}}\label{sec-trun}

Any type \(A\) can be turned into an \(n\)-type
\(\left\| A \right\|_{n}\), called the \textbf{\(n\)-truncation} of
\(A\), by a map \(| - |_{n}:A \rightarrow \left\| A \right\|_{n}\). The construction of  \(\left\| A \right\|_{n}\) as an HIT is elaborated in \Cref{app:hub}.

The induction principle (called the \textbf{truncation induction}) for
\(\left\| A \right\|_{n}\) states that given

\begin{itemize}
\item
  a dependent type
  \(P:\left\| A \right\|_{n} \rightarrow n\mathtt{\text{-}Type}\) and
\item
  a dependent function \(g:(a:A) \rightarrow P\left( |a|_{n} \right)\),
\end{itemize}
there is a dependent function
\(f:\left( x:\left\| A \right\|_{n} \right) \rightarrow P(x)\), such
that \(f\left( |a|_{n} \right): \equiv g(a)\) for all \(a:A\)
\protect{\cite[Theorem 7.3.2]{hott}}. The recursion principle states that given

\begin{itemize}
\item
  an \(n\)-type \(E\) and
\item
  a function \(g:A \rightarrow E\),
\end{itemize}
there is a function \(f:\left\| A \right\|_{n} \rightarrow E\) such that
\(f\left( |a|_{n} \right): \equiv g(a)\).

Particularly, the \(( - 1)\)-truncation turns any type into a mere
proposition. This allows us to get rid of the `superfluous' information
carried by the \(\Sigma\)-type and disjoint union, as footnoted in
\Cref{interpretation-table}, and perform logic in the classical sense. For
\(A:\Type\) and \(P:A \rightarrow \Type\), we shall use
the expression `there \textbf{merely} exists \(a:A\) such that
\(P(a)\)' to represent the type
\(\left\| {\sum_{a:A}P(a)} \right\|_{- 1}\) (see \protect{\cite[Sections 3.6, 3.7, 3.10]{hott}}).

The next result shows that truncation operation is a functor:

\begin{proposition} Given \(n \geq - 2\) and \(f:A \rightarrow B\), there
is a map
\(\left\| f \right\|_{n}:\left\| A \right\|_{n} \rightarrow \left\| B \right\|_{n}\)
such that
\(\left\| f \right\|_{n}\left( |a|_{n} \right): \equiv |f(a)|_{n}\). \label{truncation-functor} \end{proposition} \begin{proof} Since
\(\left\| B \right\|_{n}\) is an \(n\)-type, define
\(\left\| f \right\|_{n}:\left\| A \right\|_{n} \rightarrow \left\| B \right\|_{n}\)
by \(\left\| f \right\|_{n}\left( |a|_{n} \right): \equiv |f(a)|_{n}\)
by the recursion principle. \end{proof}

Unsurprisingly, \(n\)-truncating an \(n\)-type gives an equivalence:
\begin{proposition}[\protect{\cite[Corollary 7.3.7]{hott}}]
 Let
\(A:\Type\). \(A\) is an \(n\)-type if and only if the map
\(| - |_{n}:A \rightarrow \left\| A \right\|_{n}\) is an equivalence.

\label{trun-n-type} \end{proposition} \begin{proof} Suppose \(A\) is an
\(n\)-type. Then we are allowed to define
\(f:\left\| A \right\|_{n} \rightarrow A\) by
\(f\left( |a|_{n} \right): \equiv a\). By definition,
\(f\circ| - |_{n} \equiv \operatorname{id}\limits_{A}\). Then we would
like to define a dependent function
\(\left( x:\left\| A \right\|_{n} \right) \rightarrow \left( \left| {f(x)} \right|_{n} =_{\left\| A \right\|_{n}}^{}x \right)\).
Notice that
\(\left( \left| {f(x)} \right|_{n} =_{\left\| A \right\|_{n}}^{}x \right)\)
is an \(n\)-type (since it is by definition an \((n - 1)\)-type), so by
its induction principle it suffices to give a dependent function
\((a:A) \rightarrow \left( \left| {f\left( |a|_{n} \right)} \right|_{n} =_{\left\| A \right\|_{n}}^{}|a|_{n} \right)\),
which can be given by \(\lambda a.\mathtt{refl}_{|a|_{n}}\) since
\(f\left( |a|_{n} \right): \equiv a\). Therefore, \(f\) forms a
quasi-inverse for \(| - |_{n}\) and \(| - |_{n}\) is an equivalence.

Conversely, if \(| - |_{n}\) is an equivalence, then since
\(\left\| A \right\|_{n}\) is an \(n\)-type, so is \(A\). \end{proof}

The following two results deal with \(n\)-truncations of identity types
and loop spaces:

\begin{proposition}[\protect{\cite[Theorem 7.3.12]{hott}}]
 Let \(n \geq - 2\),
\(A:\Type\) and \(x,y:A\). Then
\[\left\| {x =_{A}^{}y} \right\|_{n} \simeq \left( |x|_{n + 1} =_{\left\| A \right\|_{n + 1}}^{}|y|_{n + 1} \right).\]

\label{truncation-identity} \end{proposition}
\begin{proposition}[\protect{\cite[Corollary 7.3.14]{hott}}]
 Let \(n \geq - 2\),
\(k \geq 0\), and \(A\) be a pointed type. Then
\(\left\| {\Omega^{k}(A)} \right\|_{n} \simeq \Omega^{k}\left( \left\| A \right\|_{n + k} \right).\)

\label{truncation-omega-commute} \end{proposition}

Finally, \(n\)-truncations have the cumulativity property:

\begin{lemma}[\protect{\cite[Lemma 7.3.15]{hott}}]
 For \(A:\Type\) and
\(- 2 \leq m \leq n\),
\(\left\| \left\| A \right\|_{n} \right\|_{m} \simeq \left\| A \right\|_{m}.\)

\label{trun-trun} \end{lemma}

\section{Interlude: homotopy groups of the circle}\label{sec-algebra}

In this interlude, we first take groups as an example to demonstrate how
to construct algebraic structures. Then we define \emph{homotopy groups}
and calculate them for \({\mathbb{S}}^{1}\).

\begin{definition} A \textbf{group} is a \emph{set} \(G\) with functions
\(m:G \rightarrow G \rightarrow G\) (called the \emph{multiplication}
function), \(i:G \rightarrow G\) (called the \emph{inversion} function)
and a term \(e:G\) (called the \emph{neutral element}), such that the
equalities hold for all \(x,y,z:G\):

\begin{itemize}
\item
  \(m\left( m(x)(y) \right)(z) =_{G}^{}m(x)\left( m(y)(z) \right)\);
\item
  \(m(x)\left( i(x) \right) =_{G}^{}e\);
\item
  \(m\left( i(x) \right)(x) =_{G}^{}e\);
\item
  \(m(x)(e) =_{G}^{}x\);
\item
  \(m(e)(x) =_{G}^{}x\).
\end{itemize}

The group is \textbf{abelian} if further \(m(x)(y) =_{G}^{}m(y)(x)\)
hold for all \(x,y:G\). \end{definition}

\begin{remark} Since a group is by definition a set, any of these
equalities has at most one inhabitant. Hence, the predicate
`\((G,m,i,e)\) is a group' is a mere proposition. This allows the
groups we define to behave identically as in classical mathematics. \end{remark}

\begin{definition} Let \((G,m,i,e)\) and
\((G^{\prime},m^{\prime},i^{\prime},e^{\prime})\) be two groups. Then a function
\(f:G \rightarrow G^{\prime}\) is called a \textbf{group homomorphism} if
\(m^{\prime}\left( f(x) \right)\left( f(y) \right) =_{G^{\prime}}^{}f\left( m(x) \right)\left( m(y) \right)\)
for all \(x,y:G\). A group homomorphism \(f\) is a \textbf{group
isomorphism} if it is an equivalence. \end{definition}

\begin{example} \(\mathbb{1}\) is trivially a group, written as \(0\). \end{example}

Importantly, set-truncation enables us to turn a loop space into a
group: \begin{definition} Let \(A\) be a pointed type and \(n \geq 1\).
Define the \textbf{\(n\)-th homotopy group} of \(A\) as
\(\pi_{n}(A): \equiv \left\| {\Omega^{n}(A)} \right\|_{0}\). \end{definition}

When \(n \geq 1\), the set \(\pi_{n}(A)\) is a group with path
concatenation, path inversion and reflexivity. When \(n \geq 2\), by the
Eckmann-Hilton argument \protect{\cite[Theorem 2.1.6]{hott}}, \(\pi_{n}(A)\) is an
abelian group. We may also write
\(\pi_{0}(A) \equiv \left\| A \right\|_{0}\), although it is not a group
in general. Finally, \(\pi_{n}\) is a functor, as truncation and loop
space are both functors by \Cref{loop-functor} and \Cref{truncation-functor}.

We know that \({\mathbb{S}}^{1}\) is not contractible by \Cref{s1-not-1}. For
the rest of this section, we briefly sketch the calculation of
\(\pi_{k}\left( {\mathbb{S}}^{1} \right)\) for all \(k \geq 1\). We
proceed by first showing
\(\Omega^{1}\left( {\mathbb{S}}^{1} \right) \simeq {\mathbb{Z}}\) with
the \emph{encode-decode method} based on \protect{\cite[Section 8.1.4]{hott}}. This
technique will be used again in \Cref{sec-blakers}.

The initial idea would be to give two functions
\(\Omega^{1}\left( {\mathbb{S}}^{1} \right) \leftrightarrows {\mathbb{Z}}\)
and show that they are quasi-inverses. The function
\({\mathbb{Z}} \rightarrow \Omega^{1}\left( {\mathbb{S}}^{1} \right)\)
is given by \(\varphi: \equiv \lambda n.\mathtt{loop}^{n}\), where
\(\mathtt{loop}^{n}\) denotes the \(|n|\)-time concatenation of
\(\mathtt{loop}\) if \(n > 0\) or \(\mathtt{loop}^{- 1}\) if
\(n < 0\) and \(\mathtt{refl}_{\mathtt{base}}\) if \(n = 0\).
However, we cannot easily define a function
\(\left( \mathtt{base} =_{{\mathbb{S}}^{1}}^{}\mathtt{base} \right) \rightarrow {\mathbb{Z}}\),
as path induction is unavailable when the two endpoints are both fixed.
We thus generalise the claim
\(\Omega^{1}\left( {\mathbb{S}}^{1} \right) \simeq {\mathbb{Z}}\) to the
following.

\begin{lemma}  There exists a dependent type
\(\mathtt{code}:{\mathbb{S}}^{1} \rightarrow \Type\), such
that
\[\left( x:{\mathbb{S}}^{1} \right) \rightarrow \left( \left( \mathtt{base} =_{{\mathbb{S}}^{1}}^{}x \right) \simeq \mathtt{code}(x) \right)\]
and \(\mathtt{code}(\mathtt{base}): \equiv {\mathbb{Z}}\). \label{s1-lemma} \end{lemma}

\begin{proof}[Proof sketch]
To finish the definition of
\(\mathtt{code}\), as \(\mathtt{succ} :{\mathbb{Z}} \rightarrow {\mathbb{Z}}\)
is an equivalence, \(\mathtt{ua}( \mathtt{succ} )\) is a path
\({\mathbb{Z}} = {\mathbb{Z}}\). We then set
\(\mathtt{ap}_{\mathtt{code}}\left( \mathtt{loop} \right) := \mathtt{ua}( \mathtt{succ} )\).
	\texttt{code} represents the universal cover of the
circle: the fibre over \(\mathtt{base}\) is \(\mathbb{Z}\) and going
around 	\texttt{loop} once corresponds to going up by
one in \(\mathbb{Z}\).

From this definition, we can prove that
\begin{equation}
  \mathtt{transport}^{\mathtt{code}}\left( \mathtt{loop}^{c} \right)(x) = (x + c)
\label{tr-code-loop-3} \end{equation} for any \(c:{\mathbb{Z}}\) (by
\protect{\cite[Lemma 8.1.2]{hott}} and \Cref{transport-functorality}).

It then remains to define four functions of the types:

\begin{enumerate}
\item
  \(\mathtt{encode}:\left( x:{\mathbb{S}}^{1} \right) \rightarrow \left( \mathtt{base} =_{{\mathbb{S}}^{1}}^{}x \right) \rightarrow \mathtt{code}(x)\),
\item
  \(\mathtt{decode}:\left( x:{\mathbb{S}}^{1} \right) \rightarrow \mathtt{code}(x) \rightarrow \left( \mathtt{base} =_{{\mathbb{S}}^{1}}^{}x \right)\),
\item
  \(\left( x:{\mathbb{S}}^{1} \right) \rightarrow \left( p:\mathtt{base} =_{{\mathbb{S}}^{1}}^{}x \right) \rightarrow \left( \mathtt{decode}(x)(\mathtt{encode}(x)(p)) = p \right)\),
  and
\item
  \(\left( x:{\mathbb{S}}^{1} \right) \rightarrow \left( c:\mathtt{code}(x) \right) \rightarrow \left( \mathtt{encode}(x)(\mathtt{decode}(x)(c)) = c \right)\),
\end{enumerate}
where (3) and (4) state that for each \(x:{\mathbb{S}}^{1}\),
\(\mathtt{encode}(x)\) and \(\mathtt{decode}(x)\) are
quasi-inverses.

To proceed, (1) is given by
\[\mathtt{encode}(x)(p): \equiv \mathtt{transport}^{\mathtt{code}}(p)(0);\]
(2) uses circle induction, where \(\mathtt{decode}(\mathtt{base})\)
is given by the function
\(\varphi:{\mathbb{Z}} \rightarrow \left( \mathtt{base} = \mathtt{base} \right)\)
defined earlier; (3) uses path induction and (4) circle induction, and
they both apply Equality~\ref{tr-code-loop-3}. We omit the details. 
\end{proof}

Now setting \(x\) to \(\mathtt{base}\) in \Cref{s1-lemma} yields
\(\Omega^{1}\left( {\mathbb{S}}^{1} \right) \simeq {\mathbb{Z}}\).
\begin{corollary}
\(\pi_{1}\left( {\mathbb{S}}^{1} \right) \simeq {\mathbb{Z}}\) and
\(\pi_{k}\left( {\mathbb{S}}^{1} \right) \simeq 0\) for \(k \geq 2\).
\label{pk-s1} \end{corollary} \begin{proof} Since \(\mathbb{Z}\) is a
set, \(\left\| {\mathbb{Z}} \right\|_{0} = {\mathbb{Z}}\). Thus,
\(\pi_{1}\left( {\mathbb{S}}^{1} \right) \equiv \left\| {\Omega^{1}\left( {\mathbb{S}}^{1} \right)} \right\|_{0} \simeq \left\| {\mathbb{Z}} \right\|_{0} = {\mathbb{Z}}\).
Further, since \(\mathbb{Z}\) is a set,
\(\Omega^{1}\left( {\mathbb{S}}^{1} \right)\) is also a set, which
implies that \(\Omega^{k}\left( {\mathbb{S}}^{1} \right)\) is trivial
for any \(k \geq 2\), and therefore so is
\(\pi_{k}\left( {\mathbb{S}}^{1} \right)\). \end{proof}

\section{\texorpdfstring{\(n\)-connected
types}{n-connected types}}\label{sec-connect}

\begin{definition} Let \(A,B:\Type\) and \(f:A \rightarrow B\). The
\textbf{fibre} of \(f\) over \(y:B\) is defined as the type
\[\mathtt{fib}_{f}(y): \equiv \sum_{x:A}\left( f(x) =_{B}^{}y \right).\]
\label{fibre} \end{definition}

\begin{definition} A type \(A\) is \textbf{\(n\)-connected} if
\(\left\| A \right\|_{n}\) is contractible, and a function
\(f:A \rightarrow B\) is \textbf{\(n\)-connected} if for all \(y:B\),
\(\left\| {\mathtt{fib}_{f}(y)} \right\|_{n}\) is contractible. In
particular, a function \(f:A \rightarrow B\) is \textbf{surjective},
\textbf{connected}, or \textbf{simply-connected} if it is \(( - 1)\)-,
\(0\)-, or \(1\)-connected, respectively. \end{definition} \begin{remark} An
\(n\)-connected function \(f\) can be thought of as `being surjective up
to dimension \((n + 1)\)'. In particular, \(f\) preserves `homotopy
information' up to dimension \(n\). This is formalised in
\Cref{connected-induce-equivalence} and \Cref{connect-surjection}. We also remark
that different conventions in defining the \(n\)-connectedness of
functions exist. \end{remark}

Easily seen from the definitions are:

\begin{enumerate}
\item
  A type \(A\) is \(n\)-connected if and only if the unique function
  \(A \rightarrow \mathbb{1}\) is \(n\)-connected;
\item
  Every type is \(( - 2)\)-connected and so is every function;
\item
  Every pointed type is \(( - 1)\)-connected.
\end{enumerate}

\begin{example}[\redDiamond  Inspired by \protect{\cite[Figure 3.4]{kbh}}]
Consider an HIT \(G\) generated by a non-empty
graph \(\overline{G}\): each vertex \( \overline{v} \in \overline{G}\) gives
a term \(v:G\) and each edge \(\overline{e}\) connecting vertices
\(\overline{v_{1}}\) and \(\overline{v_{2}}\) gives a path
\(e:v_{1} =_{G}^{}v_{2}\), where the order of \(v_{1}\) and \(v_{2}\) is
arbitrary thanks to path inversion.

If \(\overline{G}\) has only one connected component, then \(G\) is
\(0\)-connected. Indeed, as \(\left\| G \right\|_{0}\) is pointed, it
suffices to show that \(x^{\prime} = y^{\prime}\) is contractible for any
\(x^{\prime},y^{\prime}:\left\| G \right\|_{0}\). The goal is a mere
proposition, so we assume \(x^{\prime}\) is \(|x|_{0}\) and \(y^{\prime}\) is
\(|y|_{0}\) for \(x,y:G\). By \Cref{truncation-identity},
\(\left( |x|_{0} = |y|_{0} \right) \simeq \left\| {x = y} \right\|_{- 1}\),
but \(\left\| {x = y} \right\|_{- 1}\) is pointed by assumption and is
thus contractible.

If further \(\overline{G}\) is a tree, then \(G\) is \(1\)-connected.
Indeed, using the same reasoning twice, it suffices to show that
\(\left\| {p = q} \right\|_{- 1}\) is contractible for any \(x,y:G\) and
\(p,q:x = y\), which holds since any two vertices on a tree has a unique
path between them up to homotopy. 
\label{graph-example} \end{example}

\begin{example}\({\mathbb{S}}^{0} \equiv \mathbb{2}\) is pointed and is
\(( - 1)\)-connected. \Cref{graph-example} indicates that \({\mathbb{S}}^{1}\)
is \(0\)-connected. In fact, for all \(n:{\mathbb{N}}\), the sphere
\({\mathbb{S}}^{n}\) is \((n - 1)\)-connected \protect{\cite[Corollary
8.2.2]{hott}}. \label{s-n-connected} \end{example}

We observe that \(n\)-connected types and functions have the (downward)
cumulativity property:

\begin{proposition}[\protect{\cite[Proposition 2.3.7]{brunerie}}]
 Let
\(- 2 \leq m \leq n\). Any \(n\)-connected type (resp. function) is also
an \(m\)-connected type (resp. function). 
\label{connected-below} \end{proposition} \begin{proof} Let
\(A:\Type\) be \(n\)-connected. Then
\(\left\| A \right\|_{m} \simeq \left\| \left\| A \right\|_{n} \right\|_{m}\)
by \Cref{trun-trun}, and
\(\left\| \left\| A \right\|_{n} \right\|_{m} \simeq \left\| A \right\|_{n}\)
by \Cref{trun-n-type} since \(\left\| A \right\|_{n}\) is contractible and
\emph{a fortiori} an \(m\)-type. So \(\left\| A \right\|_{m}\) is
contractible. The case for \(n\)-connected functions is similar. \end{proof}

We can also say something about the lower homotopy groups of an
\(n\)-connected type:

\begin{lemma}[\protect{\cite[Lemma 8.3.2]{hott}}]
 If \(A\) is a pointed
\(n\)-connected type, then \(\pi_{k}(A) \simeq 0\) for all \(k \leq n\).

\label{n-connected-pi-k-1} \end{lemma} \begin{proof}We calculate
\[\pi_{k}(A) \equiv \left\| {\Omega^{k}(A)} \right\|_{0} \simeq \Omega^{k}\left( \left\| A \right\|_{k} \right) \simeq \Omega^{k}\left( \left\| \left\| A \right\|_{n} \right\|_{k} \right) \simeq \Omega^{k}\left( \left\| \mathbb{1} \right\|_{k} \right) \simeq \Omega^{k}\left( \mathbb{1} \right) \simeq \mathbb{1},\]
by \Cref{truncation-omega-commute} and \Cref{trun-trun}. \end{proof}

\begin{corollary} \(\pi_{k}\left( {\mathbb{S}}^{n} \right) \simeq 0\) for
\(k < n\). \label{pk-sn-k-less} \end{corollary} \begin{proof} By
\Cref{s-n-connected} and \Cref{n-connected-pi-k-1}. \end{proof}

We have `induction principles' for \(n\)-connected types and functions:
\begin{proposition}[Induction principle for \(n\)-connected types]
  Suppose \(n \geq - 2\) and \(A\) is an \((n + 1)\)-connected type. Given
  
  \begin{itemize}
    \item
    a dependent type \(P:A \rightarrow n\mathtt{\text{-}Type}\),
    \item
    \(a_{0}:A\) and
    \item
    \(p:P\left( a_{0} \right)\),
  \end{itemize}
  there is a dependent function \(f:(x:A) \rightarrow P(x)\) such that
  \(f\left( a_{0} \right): \equiv p\).
\label{induction-connected} \end{proposition}

\begin{proof}\redDiamond Since \(n\mathtt{\text{-}Type}\) is an
\((n + 1)\)-type, by the recursion principle of
\(\left\| A \right\|_{n + 1}\), there is a dependent type
\(P^{\prime}:\left\| A \right\|_{n + 1} \rightarrow n\mathtt{\text{-}Type}\)
such that \(P^{\prime}\left( |a|_{n + 1} \right): \equiv P(a)\). Since
\(\left\| A \right\|_{n + 1}\) is contractible, there is a function
\(f^{\prime}:\left( x:\left\| A \right\|_{n + 1} \right) \rightarrow P^{\prime}(x)\)
such that
\(f^{\prime}\left( \left| a_{0} \right|_{n + 1} \right): \equiv p:P\left( a_{0} \right) \equiv P^{\prime}\left( \left| a_{0} \right|_{n + 1} \right)\).
Therefore, we obtain
\(f: \equiv \left( f^{\prime}\circ| - |_{n + 1} \right):(a:A) \rightarrow P^{\prime}\left( |a|_{n + 1} \right) \equiv P(a)\).
\end{proof}

\begin{proposition}[`Induction principle' for \(n\)-connected maps, \protect{\cite[Lemma 7.5.7]{hott}}]
  Suppose \(n \geq - 2\),
  \(A,B:\Type\), and \(f:A \rightarrow B\). Then \(f\) is
  \(n\)-connected, if and only if for any
  \(P:B \rightarrow n\mathtt{\text{-}Type}\), the map
  \[\lambda s.s\circ f:\left( (b:B) \rightarrow P(b) \right) \rightarrow \left( (a:A) \rightarrow P\left( f(a) \right) \right)\]
  is an equivalence. 
\label{connected-map-converse} \end{proposition}

Intuitively, \Cref{induction-connected} says that an \((n + 1)\)-connected
type `looks like' a contractible type to any family of \(n\)-types, and
\Cref{connected-map-converse} says that in the presence of an \(n\)-connected
map \(f:A \rightarrow B\), a family of \(n\)-types over \(B\) sees \(B\)
as being `generated' by \(A\) \protect{\cite[Section 2]{blakers}}.

Finally, we mention that an \(n\)-connected function induces an
equivalence on \(n\)-truncations:

\begin{lemma}[\protect{\cite[Lemma 7.5.14]{hott}}]
 If \(f:A \rightarrow B\) is
\(n\)-connected, then it induces an equivalence
\(\left\| A \right\|_{n} \simeq \left\| B \right\|_{n}\). 
\label{connected-induce-equivalence} \end{lemma}

\chapter{Hopf fibration}\label{sec-homotopy}

The centre of this chapter is the following definition:
\begin{definition}Let \(X\), \(Y\) be pointed types. Define the type of
\textbf{pointed maps} from \(X\) to \(Y\) as
\[\mathtt{Map}\hspace{0pt}_{\star}(X,Y): \equiv \sum_{f:X \rightarrow Y}\left( f\left( \star_{X} \right) =_{Y}^{} \star_{Y} \right).\]
We may say that \(f:X \rightarrow_{\star}Y\) is a pointed map
\emph{witnessed} by path
\(f_{\star}:f\left( \star_{X} \right) =_{Y}^{} \star_{Y}\). \end{definition}

For any pointed map \(f:X \rightarrow_{\star}Y\), the homotopy groups of
\(F \equiv \mathtt{fib}_{f}\left( \star_{Y} \right)\), \(X\) and \(Y\)
are related by a \emph{long exact sequence}, obtained from a \emph{fibre
sequence}, shown in \Cref{sec-fib-seq}. In particular, given a `pointed
fibration' (in the sense of \Cref{fib-fibration}), \(F\), \(X\), \(Y\)
correspond respectively to the fibre, total space and base space of the
fibration. In \Cref{sec-hopf-cons}, we work to build the \emph{Hopf
construction}, a fibration based on \emph{H-spaces}. In \Cref{sec-hopf-fib},
we demonstrate an H-space structure on \({\mathbb{S}}^{1}\) and derive
the \emph{Hopf fibration}, whose long exact sequence fascinatingly
connects the homotopy groups of \({\mathbb{S}}^{1}\),
\({\mathbb{S}}^{2}\) and \({\mathbb{S}}^{3}\).

\section{Fibre sequences}\label{sec-fib-seq}

We begin with a simple yet powerful observation.

\begin{lemma}  Let \(X\) and \(Y\) be pointed types and let
\(f:X \rightarrow_{\star}Y\) be witnessed by \(f_{\star}\). Let
\(F: \equiv \mathtt{fib}_{f}\left( \star_{Y} \right)\) and
\(i \equiv \pr_{1}:F \rightarrow X\). Then we have pointed
types and pointed maps:
\[F\overset{i}{\rightarrow}X\overset{f}{\rightarrow}Y.\] \label{fib-pointed-type} \end{lemma} 
\begin{proof} \(F\) is pointed by
\(\left( \star_{X},f_{\star} \right)\), and since
\(i\left( \star_{X},f_{\star} \right) \equiv \star_{X}\), \(i\) is a
pointed map witnessed by \(\mathtt{refl}_{\star_{X}}\). \end{proof}

A special case is the following, a fibration over \(Y\) with fibre \(F\)
and total space \(X\).

\begin{corollary} Let \(Y\) be a pointed type and
\(P:Y \rightarrow \Type\) be a dependent type such that
\(P\left( \star_{Y} \right)\) is a pointed type. Then we have pointed
types and pointed maps:
\[P\left( \star_{Y} \right)\overset{i}{\rightarrow}\sum_{y:Y}P(y)\overset{f}{\rightarrow}Y,\]
called the \textbf{fibration sequence}, where
\(f \equiv \pr_{1}\) and
\(i \equiv \lambda z.\left( \star_{Y},z \right)\). \label{fib-fibration} \end{corollary} \begin{proof}\redDiamond
\(X \equiv \sum_{y:Y}P(y)\) has basepoint
\(\left( \star_{Y}, \star_{P\left( \star_{Y} \right)} \right)\) and
\(f\) is thus a pointed map. Then we can apply \Cref{fib-pointed-type}, but
\(F \equiv \mathtt{fib}_{f}\left( \star_{Y} \right) \simeq P\left( \star_{Y} \right)\)
by \protect{\cite[Lemma 4.8.1]{hott}} (which uses \Cref{sigma-id-contractible}). The map
\(i\) is as above under this equivalence. \end{proof}

The rest of this section is based on the more general \Cref{fib-pointed-type},
although later in practice we only use the case \Cref{fib-fibration}. Since
\(i:F \rightarrow X\) is again a pointed map, we apply \Cref{fib-pointed-type}
recursively to yield the following construction:

\begin{definition} Let \(X\) and \(Y\) be pointed types. The \textbf{fibre
sequence} of a pointed map \(f:X \rightarrow_{\star}Y\) witnessed by
\(f_{\star}\) is the infinite sequence of pointed types and pointed maps
\[\ldots\overset{f^{(n)}}{\rightarrow}X^{(n)}\overset{f^{(n - 1)}}{\rightarrow}X^{(n - 1)}\overset{f^{(n - 2)}}{\rightarrow}X^{(n - 2)} \rightarrow \ldots\overset{f^{(0)}}{\rightarrow}X^{(0)}.\]

We denote \(\star^{(n)}:X^{(n)}\) as the basepoint of \(X^{(n)}\) and
\(f_{\star}^{(n)}:f^{(n)}\left( \star^{(n + 1)} \right) =_{X^{(n)}}^{} \star^{(n)}\)
as the witnessing path of \(f^{(n)}\). The sequence is then defined
recursively by

\begin{itemize}
\item
  \(X^{(0)}: \equiv Y\), \(X^{(1)}: \equiv X\), \(f^{(0)}: \equiv f\),
  \(\star^{(0)}: \equiv \star_{Y}\), \(\star^{(1)}: \equiv \star_{X}\),
  \(f_{\star}^{(0)}: \equiv f_{\star}\);
\item
  If \(n \geq 2\), \[\begin{aligned}
  X^{(n)}: \equiv \mathtt{fib}_{f^{(n - 2)}}\left( \star^{(n - 2)} \right) & ,\quad f^{(n - 1)}: \equiv \pr_{1}:X^{(n)} \rightarrow X^{(n - 1)}, \\
   \star^{(n)}: \equiv \left( \star^{(n - 1)},f_{\star}^{(n - 2)} \right) & ,\quad f_{\star}^{(n - 1)}: \equiv \mathtt{refl}_{\star^{(n - 1)}}.
  \end{aligned}\]
\end{itemize}

We denote
\(F: \equiv \mathtt{fib}_{f}\left( \star_{Y} \right) \equiv X^{(2)}\).
\end{definition}

This fibre sequence can be (magically) turned into an equivalent form
where the loop spaces of \(F\), \(X\) and \(Y\) show up in a periodic
fashion. The following two ingredients are needed before we move on.

\begin{definition} If \(f:A \rightarrow_{\star}B\) is a type equivalence,
then \(f\) is a \textbf{pointed equivalence} and in this case \(A\) and
\(B\) are \textbf{pointed equivalent}. \end{definition} \begin{lemma}  The
\textbf{looping} of \(f:A \rightarrow_{\star}B\) defined as
\(\Omega f: \equiv \left( \lambda q.f_{\star}^{- 1}\ct\mathtt{ap}_{f}(q)\ct f_{\star} \right):\Omega A \rightarrow \Omega B\)
is also a pointed map. \label{loop-functor} \end{lemma}
\begin{proof}
\((\Omega f)\left( \mathtt{refl}_{\star_{A}} \right) \equiv f_{\star}^{- 1}\ct\mathtt{refl}_{f\left( \star_{A} \right)}\ct f_{\star} = \mathtt{refl}_{\star_{B}}\).
\end{proof}

We now work to reveal the first step of the alleged equivalence.

\begin{lemma}[\protect{\cite[Lemma 8.4.4]{hott}}]
 Let \(X\) and \(Y\) be pointed
types and let \(f:X \rightarrow_{\star}Y\). Then
\(X^{(3)} \simeq \Omega Y\) in the fibre sequence of \(f\). Further,
under the equivalence, \(f^{(2)}\) is identified with
\[\partial: \equiv \lambda r.\left( \star_{X},f_{\star}\ct r \right):\Omega Y \rightarrow F.\]

\label{X3-omega-Y} \end{lemma} 

\begin{figure}[h]
    \centering
    \figFourOneSeven
    \caption{Diagram for \Cref{X3-omega-Y}.}
    \label{fig-417}
\end{figure}

\begin{proof}See \Cref{fig-417}. Using the
definition of fibres twice, a term of \(X^{(3)}\) is of the form
\(\left( (x,p),q \right)\), where \(x:X\), \(p:f(x) =_{Y}^{} \star_{Y}\)
and \(q:x =_{X}^{} \star_{X}\). We first define a function
\(\varphi:X^{(3)} \rightarrow \Omega(Y)\). Take any
\(\left( (x,p),q \right):X^{(3)}\). Then
\(\mathtt{ap}_{f}\left( q^{- 1} \right):f\left( \star_{X} \right) =_{Y}^{}f(x)\).
Thus since \(f_{\star}:f\left( \star_{X} \right) =_{Y}^{} \star_{Y}\),
we can produce
\(f_{\star}^{- 1}\ct\mathtt{ap}_{f}\left( q^{- 1} \right)\ct p:\Omega(Y)\).
Conversely, we define \(\psi:\Omega(Y) \rightarrow X^{(3)}\) by
\[\lambda r.\left( \left( \star_{X},f_{\star}\ct r \right),\mathtt{refl}_{\star_{X}} \right).\]
Now take any \(\left( (x,p),q \right):X^{(3)}\), then \[\begin{aligned}
\psi(\varphi((x,p),q)) & \equiv \psi\left( f_{\star}^{- 1}\ct\mathtt{ap}_{f}\left( q^{- 1} \right)\ct p \right) \\
 & = \left( \left( \star_{X},\mathtt{ap}_{f}\left( q^{- 1} \right)\ct p \right),\mathtt{refl}_{\star_{X}} \right) \\
 & = \left( (x,p),q \right),
\end{aligned}\] where the last equality is by \Cref{sigma-path} and
\Cref{identity-path-1}. Conversely, for any \(r:\Omega(Y, \star_{Y})\),
\[\varphi(\psi(r)) \equiv \varphi(\left( \star_{X},f_{\star}\ct r \right),\mathtt{refl}_{\star_{X}}) = f_{\star}^{- 1}\ct\mathtt{ap}_{f}\left( \mathtt{refl}_{\star_{X}}^{- 1} \right)\ct f_{\star}\ct r = r.\]
Thus \(\psi\) forms a quasi-inverse of \(\varphi\). Also \(\psi\) and
\(\varphi\) are both pointed maps, so \(X^{(3)}\) and \(\Omega Y\) are
pointed equivalent. We therefore identify
\(f^{(2)}:X^{(3)} \rightarrow F\) with
\(\partial: \equiv f^{(2)}\circ\psi:\Omega Y \rightarrow F\), which can
be simplified as
\(\lambda r.\left( \star_{X},f_{\star}\ct r \right)\), easily a
pointed map. \end{proof}

By applying \Cref{X3-omega-Y} recursively to the fibre sequence, we have:

\begin{corollary}[\protect{\cite[Lemma 8.4.4]{hott}}]
 The fibre sequence of
\(f:X \rightarrow_{\star}Y\) can be equivalently written as
\[\rightarrow \Omega^{2}X\overset{\Omega^{2}f}{\rightarrow}\Omega^{2}Y\overset{- \Omega\partial}{\rightarrow}\Omega F\overset{- \Omega i}{\rightarrow}\Omega X\overset{- \Omega f}{\rightarrow}\Omega Y\overset{\partial}{\rightarrow}F\overset{i}{\rightarrow}X\overset{f}{\rightarrow}Y.\]
where \(i: \equiv \pr_{1}\) and the minus sign represents
composition with path inversion. 
\label{fibre-sequence-2} \end{corollary} We will then show that this
fibre sequence, when set-truncated, becomes an \emph{exact sequence},
defined below. \begin{definition} Let \(A,B\) be sets and let
\(f:A \rightarrow B\), then the \textbf{image} of \(f\) is defined as
\(\operatorname{im}(f): \equiv \sum_{b:B}\left\| {\mathtt{fib}_{f}(b)} \right\|_{- 1}.\)
If \(A,B\) are pointed sets and let \(f:A \rightarrow_{\star}B\), then
the \textbf{kernel} of \(f\) is defined as
\(\ker(f): \equiv \mathtt{fib}_{f}\left( \star_{B} \right)\). \end{definition}

\begin{definition} An \textbf{exact sequence of pointed sets} is a sequence
of pointed sets and pointed maps:
\[\ldots \rightarrow A^{(n + 1)}\overset{f^{(n)}}{\rightarrow}A^{(n)}\overset{f^{(n - 1)}}{\rightarrow}A^{(n - 1)} \rightarrow \ldots\]
such that \(\operatorname{im}(f^{(n)}) = \ker(f^{(n - 1)})\) for every
\(n\). The sequence is an \textbf{exact sequence of groups} if further
each \(A^{(n)}\) is a group and each \(f^{(n)}\) is a group
homomorphism. \end{definition}

\begin{lemma}[\protect{\cite[Theorem 8.4.6]{hott}} \footnote{\redDiamond A few typos in \protect{\cite[Theorem 8.4.6]{hott}} are fixed in this proof.}]
For
any \(F\overset{i}{\rightarrow}X\overset{f}{\rightarrow}Y\) given by
\Cref{fib-pointed-type},
\[\left\| F \right\|_{0}\overset{\left\| i \right\|_{0}}{\rightarrow}\left\| X \right\|_{0}\overset{\left\| f \right\|_{0}}{\rightarrow}\left\| Y \right\|_{0}\]
is an exact sequence of pointed sets. 
\label{les-lemma} \end{lemma} \begin{proof}
\(\operatorname{im}(\left\| i \right\|_{0}) \subset \ker(\left\| f \right\|_{0})\).
Take any \((x,p):F\), then
\(f\left( i(x,p) \right) \equiv f(x) =_{Y}^{} \star_{Y}\) as witnessed
by \(p\). Thus \(\left\| {f\circ i} \right\|_{0}\) sends everything to
\(\left| \star_{Y} \right|_{0}\), and so does
\(\left\| f \right\|_{0}\circ\left\| i \right\|_{0}\) by the
functoriality of \(\left\| - \right\|_{0}\).

\(\ker(\left\| f \right\|_{0}) \subset \operatorname{im}(\left\| i \right\|_{0})\).
For any \(x^{\prime}:\left\| X \right\|_{0}\) and
\(p^{\prime}:\left\| f \right\|_{0}(x^{\prime}) =_{\left\| Y \right\|_{0}}^{}\left| \star_{Y} \right|_{0}\),
we claim that there merely exists \(t:F\) such that
\(\left| {i(t)} \right|_{0} =_{\left\| X \right\|_{0}}^{}x^{\prime}\). As
the goal is a mere proposition and \emph{a fortiori} a \(0\)-type, by
truncation induction, it suffices to assume \(x^{\prime}\) is \(|x|_{0}\)
for some \(x:X\). Then since
\(\left\| f \right\|_{0}\left( |x|_{0} \right) \equiv \left| {f(x)} \right|_{0}\),
we have
\(p^{\prime}:\left| {f(x)} \right|_{0} =_{\left\| Y \right\|_{0}}^{}\left| \star_{Y} \right|_{0}\),
which by \Cref{truncation-identity}, corresponds to some
\(p^{\prime \prime}:\left\| {f(x) =_{Y}^{} \star_{Y}} \right\|_{- 1}\).
Again by truncation induction, it suffices to assume \(p^{\prime \prime}\)
is \(|p|_{- 1}\) for some \(p:f(x) =_{Y}^{} \star_{Y}\). Thus we take
\(t: \equiv (x,p):F\), and then
\(\left| {i(t)} \right|_{0} \equiv |x|_{0} \equiv x^{\prime}\). \end{proof}

\begin{theorem}[\protect{\cite[Theorem 8.4.6]{hott}}]
   Let \(X,Y\) be pointed
types and let \(f:X \rightarrow_{\star}Y\). Let
\(F: \equiv \mathtt{fib}_{f}\left( \star_{Y} \right)\). Then we have
an exact sequence of pointed sets \[\begin{aligned}
\ldots \rightarrow & \pi_{k}(F) \rightarrow \pi_{k}(X) \rightarrow \pi_{k}(Y) \rightarrow \\
\ldots \rightarrow & \pi_{2}(F) \rightarrow \pi_{2}(X) \rightarrow \pi_{2}(Y) \\
 \rightarrow & \pi_{1}(F) \rightarrow \pi_{1}(X) \rightarrow \pi_{1}(Y) \\
 \rightarrow & \pi_{0}(F) \rightarrow \pi_{0}(X) \rightarrow \pi_{0}(Y).
\end{aligned}\] This is further an exact sequence of groups except the
\(\pi_{0}\) terms, and an exact sequence of abelian groups except the
\(\pi_{0}\) and \(\pi_{1}\) terms.   \label{fibre-les} \end{theorem}

\begin{proof} By recursively applying \Cref{les-lemma} to the fibre sequence of
\(f\), the above is an exact sequence of pointed sets. The rest of the
proof, including some technicalities involving path inversion, is
omitted. \end{proof}

The following lemma is commonly known in classical algebraic topology.

\begin{lemma}[\protect{\cite[Lemma 8.4.7]{hott}}]
 Let
\(C \rightarrow A\overset{f}{\rightarrow}B \rightarrow D\) be an exact
sequence of abelian groups. Then (1) \(f\) is injective if \(C\) is
\(0\); (2) \(f\) is surjective if \(D\) is \(0\); (3) \(f\) is an
isomorphism if \(C\) and \(D\) are both \(0\). 
\label{exact-iso} \end{lemma}

From the next result, which characterises an \(n\)-connected function in
terms of homotopy groups, we formalise the idea that an \(n\)-connected
function is a generalisation of a surjective function. \begin{corollary}
Let \(X,Y\) be pointed types and let \(f:X \rightarrow_{\star}Y\) be
\(n\)-connected.

\begin{enumerate}
\item
  If \(k \leq n\), then \(\pi_{k}(f):\pi_{k}(X) \rightarrow \pi_{k}(Y)\)
  is an isomorphism;
\item
  If \(k = n + 1\), then
  \(\pi_{k}(f):\pi_{k}(X) \rightarrow \pi_{k}(Y)\) is surjective.
\end{enumerate}

\label{connect-surjection} \end{corollary}

\begin{proof}We omit the case when \(k = 0\) and suppose \(k > 0\). We have
an exact sequence
\[\pi_{k}(F) \rightarrow \pi_{k}(X) \rightarrow \pi_{k}(Y) \rightarrow \pi_{k - 1}(F).\]
Since \(f\) is \(n\)-connected, by \Cref{connected-below}, \(f\) is also
\(k\)-connected when \(k \leq n\). Thus
\(\left\| F \right\|_{k} \equiv \left\| {\mathtt{fib}_{f}\left( \star_{Y} \right)} \right\|_{k}\)
is contractible, and
\(\pi_{k}(F) \equiv \left\| {\Omega^{k}(F)} \right\|_{0} \simeq \Omega^{k}\left( \left\| F \right\|_{k} \right)\)
(by \Cref{truncation-omega-commute}) is also contractible for \(k \leq n\).
Applying \Cref{exact-iso} yields the result. \end{proof}

\begin{remark} Set \(Y\) to \(\mathbb{1}\) and \Cref{connect-surjection} would
imply \Cref{n-connected-pi-k-1}. \end{remark}

\section{Hopf construction}
\label{sec-hopf-cons} To trigger all the machinery in
the previous section, we only need a fibration sequence as in
\Cref{fib-fibration}. We now build the \emph{Hopf construction}, a fibration
sequence derived from a \(0\)-connected \emph{H-space}.

\begin{definition} A type \(A\) is an \textbf{H-space} if it is equipped
with a base point \(e:A\), a binary operation
\(\mu:A \times A \rightarrow A\) and equalities \(\mu(e,a) = a\) and
\(\mu(a,e) = a\) for every \(a:A\). \end{definition}

\begin{lemma}  Let \(A\) be a \(0\)-connected H-space. Then
\(\mu(a, - ):A \rightarrow A\) and \(\mu( - ,a):A \rightarrow A\) are
equivalences for every \(a:A\). \end{lemma}

\begin{proof} Let \(P:A \rightarrow \mathtt{Prop}\) be the dependent
proposition which sends \(a:A\) to the predicate that \(\mu(a, - )\)
(resp. \(\mu( - ,a)\)) is an equivalence. By \Cref{induction-connected}, it
suffices to prove \(P(e)\), which holds by definition. \end{proof}

The \emph{flattening lemma}, used in the next proof, is a result
concerning the total space of a fibration over an HIT constructed via
univalence \protect{\cite[Section 6.12]{hott}}.

\begin{theorem}[\protect{\cite[Lemma 8.5.7]{hott}}]
Let \(A\) be a
\(0\)-connected \(H\)-space. Then there is a dependent type (fibration)
\(P:\Sigma A \rightarrow \Type\) with fibre \(A\) and total
space \(\sum_{x:\Sigma A}P(x) \simeq A \ast A\), called the \textbf{Hopf
construction}. \label{hopf-construction} \end{theorem}

\begin{proof} We define a dependent type
\(P:\Sigma A \rightarrow \Type\) by
\(P\left( \mathtt{north} \right): \equiv A\),
\(P\left( \mathtt{south} \right): \equiv A\), and
\(\mathtt{ap}_{P}\left( \mathtt{merid}(a) \right):= \mathtt{ua}(\mu(a, - ))\)
since each \(\mu(a, - )\) is an equivalence \(A \rightarrow A\). By the
flattening lemma (see \protect{\cite[Lemma 8.5.3]{hott}} or \protect{\cite[Proposition 1.9.2]{brunerie}}), 
\(\sum_{x:\Sigma A}P(x)\) is equivalent to the pushout of the
span
\(A\overset{\mathtt{snd}}{\leftarrow}A \times A\overset{\mu}{\rightarrow}A.\)
We define a function \(\alpha:A \times A \rightarrow A \times A\) by
\(\alpha(x,y): \equiv \left( \mu(x,y),y \right)\), and \(\alpha\) is an
equivalence because it has quasi-inverse
\(\lambda(u,v).\left( \left( \mu( - ,v) \right)^{- 1}(u),v \right)\).
Therefore, we have an equivalence of spans 
\[\begin{tikzcd}
	A & {A \times A} & A \\
	A & {A \times A} & A,
	\arrow["{\mathtt{snd}}"', from=1-2, to=1-1]
	\arrow["\mu", from=1-2, to=1-3]
	\arrow["{\mathtt{snd}}", from=2-2, to=2-1]
	\arrow["{\mathtt{fst}}"', from=2-2, to=2-3]
	\arrow["{\mathtt{id}_A}"', from=1-1, to=2-1]
	\arrow["\alpha", from=1-2, to=2-2]
	\arrow["{\mathtt{id}_A}", from=1-3, to=2-3]
\end{tikzcd}\]
since both squares commute and each vertical line is an equivalence.
Thus, the pushout of the two spans are also equivalent, but the pushout
of the lower span is by definition \(A \ast A\). \end{proof}

\section{Hopf fibration}\label{sec-hopf-fib}

We finally present the Hopf construction induced by the H-space
structure on \({\mathbb{S}}^{1}\).

\begin{lemma}[\protect{\cite[Lemma 8.5.8]{hott}}]
 There is an \(H\)-space
structure on \({\mathbb{S}}^{1}\). 
\label{h-space-s1} \end{lemma} \begin{proof} Take
\(e: \equiv \mathtt{base}:{\mathbb{S}}^{1}\). Now we define
\(\mu:{\mathbb{S}}^{1} \rightarrow {\mathbb{S}}^{1} \rightarrow {\mathbb{S}}^{1}\)
by circle induction on the first argument. We define
\(\mu(\mathtt{base}): \equiv \operatorname{id}\limits_{{\mathbb{S}}^{1}}\);
then we need a path
\(p:\operatorname{id}\limits_{{\mathbb{S}}^{1}} =_{{\mathbb{S}}^{1} \rightarrow {\mathbb{S}}^{1}}^{}\operatorname{id}\limits_{{\mathbb{S}}^{1}}\)
such that \(\mathtt{ap}_{\mu}\left( \mathtt{loop} \right) := p\). By
function extensionality, this is equivalent to a path
\(p_{x}:x =_{{\mathbb{S}}^{1}}^{}x\) for each \(x:{\mathbb{S}}^{1}\). By
induction on \(x:{\mathbb{S}}^{1}\) again, we define
\(p_{\mathtt{base}}: \equiv \mathtt{loop}\). It then remains to give
a dependent path
\(\mathtt{loop} =_{\mathtt{loop}}^{\lambda x.x =_{{\mathbb{S}}^{1}}^{}x}\mathtt{loop}\),
which is equivalent to
\(\mathtt{loop}^{- 1}\ct\mathtt{loop}\ct\mathtt{loop} = \mathtt{loop}\)
by \Cref{identity-path-1}.

Now \(\mu(\mathtt{base})(x) \equiv x\) since
\(\mu(\mathtt{base}) \equiv \operatorname{id}\limits_{{\mathbb{S}}^{1}}\).
We show that \(\mu(x)(\mathtt{base}) =_{{\mathbb{S}}^{1}}^{}x\) by
circle induction on \(x:{\mathbb{S}}^{1}\). When \(x\) is
\(\mathtt{base}\),
\(\mu(\mathtt{base})(\mathtt{base}) \equiv \mathtt{base}\), so we
take \(\mathtt{refl}_{\mathtt{base}}\); when \(x\) `varies along'
	\texttt{loop}, we need a dependent path
\[\mathtt{transport}^{\lambda x.\mu(x)(\mathtt{base}) =_{{\mathbb{S}}^{1}}^{}x}\left( \mathtt{loop} \right)\left( \mathtt{refl}_{\mathtt{base}} \right) = \mathtt{refl}_{\mathtt{base}},\]
which is equivalent to
\(\mathtt{ap}_{ \lambda x.\mu(x)(\mathtt{base}) } \left( \mathtt{loop} \right)^{- 1}\, \ct \, \mathtt{refl}_{\mathtt{base}}\ct\mathtt{loop} = \mathtt{refl}_{\mathtt{base}}\)
by \Cref{identity-path-1}. By function extensionality,
\[\mathtt{ap}_{ \lambda x.\mu(x)(\mathtt{base}) } \left( \mathtt{loop} \right) = \mathtt{happly}(\mathtt{ap}_{\mu}\left( \mathtt{loop} \right))\left( \mathtt{base} \right) = p_{\mathtt{base}} \equiv \mathtt{loop}.\]
Finally, we change \(\mu\) to its uncurried form. \end{proof}

The next result follows from the \(3 \times 3\)-lemma
\protect{\cite[Section 1.8]{brunerie}}. 

\begin{lemma}[\protect{\cite[Proposition 1.8.8]{brunerie}}]
  For every \(m,n:{\mathbb{N}}\),
\({\mathbb{S}}^{m} \ast {\mathbb{S}}^{n} \simeq {\mathbb{S}}^{m + n + 1}\).
\label{join-sphere} \end{lemma}

\begin{theorem}
  There is a fibration, called the \textbf{Hopf fibration},
over \({\mathbb{S}}^{2}\) with fibre \({\mathbb{S}}^{1}\) and total
space \({\mathbb{S}}^{3}\).
\label{hopf-fibration} \end{theorem}

\begin{proof} \({\mathbb{S}}^{1}\) is \(0\)-connected by \Cref{s-n-connected}. By
\Cref{h-space-s1} and \Cref{hopf-construction}, there is a fibration over
\(\Sigma{\mathbb{S}}^{1} \equiv {\mathbb{S}}^{2}\) with fibre
\({\mathbb{S}}^{1}\) and total space
\({\mathbb{S}}^{1} \ast {\mathbb{S}}^{1} \simeq {\mathbb{S}}^{3}\) by
\Cref{join-sphere}. \end{proof}

We put all the pieces together by writing the long exact sequence
induced by the Hopf fibration:

\begin{corollary}[\protect{\cite[Corollary 8.5.2]{hott}}]

\(\pi_{2}\left( {\mathbb{S}}^{2} \right) \simeq {\mathbb{Z}}\) and
\(\pi_{k}\left( {\mathbb{S}}^{3} \right) \simeq \pi_{k}\left( {\mathbb{S}}^{2} \right)\)
for \(k \geq 3\). 
\label{pi2-pi3} \end{corollary} \begin{proof} By
\Cref{fib-fibration}, \Cref{fibre-les} and \Cref{hopf-fibration}, we have an exact
sequence of groups \[\begin{aligned}
\ldots & \rightarrow \pi_{k}\left( {\mathbb{S}}^{1} \right) \rightarrow \pi_{k}\left( {\mathbb{S}}^{3} \right) \rightarrow \pi_{k}\left( {\mathbb{S}}^{2} \right) \rightarrow \\
\ldots & \rightarrow \pi_{2}\left( {\mathbb{S}}^{1} \right) \rightarrow \pi_{2}\left( {\mathbb{S}}^{3} \right) \rightarrow \pi_{2}\left( {\mathbb{S}}^{2} \right) \\
 & \rightarrow \pi_{1}\left( {\mathbb{S}}^{1} \right) \rightarrow \pi_{1}\left( {\mathbb{S}}^{3} \right) \rightarrow \pi_{1}\left( {\mathbb{S}}^{2} \right).
\end{aligned}\] By \Cref{pk-s1} and \Cref{pk-sn-k-less}, this is reduced to \[
  \begin{aligned}
    ... &\to 0 \to \pi_k (\mathbb{S}^3) \to \pi_k (\mathbb{S}^2) \to \\
    ... &\to 0 \to 0 \to \pi_2 (\mathbb{S}^2) \\
        &\to \mathbb{Z} \to 0 \to 0.
  \end{aligned}
\] Applying
\Cref{exact-iso} yields the result. \end{proof}

\chapter{Blakers--Massey theorem}\label{sec-blakers}
 In this chapter, we will
`informalise' the mechanised Agda proof of the Blakers--Massey theorem
in \protect{\cite{blakers}} (or \protect{\cite[Sections 3.4, 4.3]{kbh}}) by combining
\protect{\cite[Section 8.6]{hott}}. For this, a generalised version of pushouts is
first presented in \Cref{sec-gen-po}. Then the theorem is stated, generalised
and relaxed in \Cref{sec-state}. The proof employs a more sophisticated
\emph{encode-decode method}, where we define
	\texttt{code} by the \emph{wedge connectivity lemma} in
\Cref{sec-def-code} and then show the contractibility of
	\texttt{code} in \Cref{sec-contra}. Some consequences,
including the Freudenthal suspension theorem and stability for spheres,
are derived in \Cref{sec-consequence}.

Throughout the chapter, we shall omit proofs on coherence operations too
technical for presentation, and write a square \(\square\) for an
unspecified coherence path.

\section{Generalised pushouts}\label{sec-gen-po}

\begin{lemma}[\protect{\cite[Lemma 4.8.2]{hott}}]
 For any function
\(f:A \rightarrow B\), \(A \simeq \sum_{b:B}\mathtt{fib}_{f}(b)\). 
\end{lemma}
\begin{proof}By \Cref{sigma-id-contractible}.\end{proof}

This result shows that we can interchange function types and families of
types, i.e., any function \(f:A \rightarrow B\) is equivalent to the
projection
\(\pr_{1}:\sum_{b:B}\mathtt{fib}_{f}(b) \rightarrow B\).

\redDiamond Recall that \Cref{sec-pushout} defines a pushout for
the span \(X\overset{f}{\leftarrow}C\overset{g}{\rightarrow}Y\). By
rewriting the functions as families of types and expanding the
definition of fibres, this span is equivalent to
\[X\overset{\pr_{X}}{\leftarrow}\sum_{x:X}\sum_{y:Y}\sum_{c:C}\left( \left( g(c) =_{Y}^{}y \right) \times \left( f(c) =_{X}^{}x \right) \right)\overset{\pr_{Y}}{\rightarrow}Y.\]
If we define a dependent type
\(Q:X \rightarrow Y \rightarrow \Type\), then this span can be
generalised as
\[X\overset{\pr_{X}}{\leftarrow}\sum_{x:X}\sum_{y:Y}Q(x)(y)\overset{\pr_{Y}}{\rightarrow}Y.\]

\begin{definition} Let \(X,Y:\Type\) and
\(Q:X \rightarrow Y \rightarrow \Type\). The
\textbf{(generalised) pushout} \(X \sqcup^{Q}Y\) is the HIT generated by
two point constructors \(\mathtt{inl}:X \rightarrow X \sqcup^{Q}Y\)
and \(\mathtt{inr}:Y \rightarrow X \sqcup^{Q}Y\) and a path
constructor
\[\mathtt{glue}:(x:X) \rightarrow (y:Y) \rightarrow Q(x)(y) \rightarrow \left( \mathtt{inl}(x) =_{X \sqcup^{Q}Y}^{}\mathtt{inr}(y) \right),\]
as visualised by \Cref{fig-gen-po}. \end{definition} 

\begin{figure}
    \centering
    \figPoTwo
    \caption{Diagram for the generalised pushout \(X \sqcup^{Q}Y\).}
    \caption*{\small \(X\) and \(Y\) are thought of as two `axes' of a `coordinate system', and each pair \((x,y):X \times Y\) corresponds to a type \(Q(x)(y)\). The two dashed circles represent the `image' of \(X\) and \(Y\) under \(\mathtt{inl}\) and \(\mathtt{inr}\), respectively.}
    \label{fig-gen-po}
\end{figure}

By previous discussions, when
\(Q \equiv \lambda x.\lambda y.\sum_{c:C}\left( \left( g(c) =_{Y}^{}y \right) \times \left( f(c) =_{X}^{}x \right) \right)\),
this definition is equivalent to the one in \Cref{sec-pushout}. Henceforth, we
only use the generalised pushout.

The induction principle states that given

\begin{itemize}
\item
  a dependent type \(P:X \sqcup^{Q}Y \rightarrow \Type\),
\item
  two dependent functions
  \(h_{\mathtt{inl}}:(x:X) \rightarrow P\left( \mathtt{inl}(x) \right)\)
  and
\item
  \(h_{\mathtt{inr}}:(y:Y) \rightarrow P\left( \mathtt{inr}(y) \right)\)
  and
\item
  a family of dependent paths
\end{itemize}
\[h_{\mathtt{glue}}:(x:X) \rightarrow (y:Y) \rightarrow \left( q:Q(x)(y) \right) \rightarrow h_{\mathtt{inl}}(x) =_{\mathtt{glue}_{x,y} (q)}^{P}h_{\mathtt{inr}}(y),\]
there is a dependent function
\(h:\left( z:X \sqcup^{Q}Y \right) \rightarrow P(z)\) such that
\(h\left( \mathtt{inl}(x) \right): \equiv h_{\mathtt{inl}}(x)\),
\(h\left( \mathtt{inr}(y) \right): \equiv h_{\mathtt{inr}}(y)\), and
\(\mathtt{apd}_{h}\left( \mathtt{glue}_{x,y}(q) \right) := h_{\mathtt{glue}}(x)(y)(q)\).

Henceforth, we may omit the first two parameters of \(\mathtt{glue}\)
for conciseness.

\section{Theorem statement}\label{sec-state}

\begin{theorem}[Blakers--Massey theorem]
  Let \(X,Y:\Type\)
  and \(Q:X \rightarrow Y \rightarrow \Type\). Let
  \(m,n \geq - 1\). Suppose the type \(\sum_{y:Y}Q(x)(y)\) is
  \(m\)-connected for each \(x:X\), and the type \(\sum_{x:X}Q(x)(y)\) is
  \(n\)-connected for each \(y:Y\). Then for each \(x:X\) and each
  \(y:Y\), the map
  \[\mathtt{glue}_{x,y}:Q(x)(y) \rightarrow \left( \mathtt{inl}(x) =_{X \sqcup^{Q}Y}^{}\mathtt{inr}(y) \right)\]
  is \((m + n)\)-connected.  \
\label{blakers-thm} \end{theorem}

Let \(X,Y,Q,m,n\) be fixed. We also fix \(x \equiv x_{0}\). By the
definition of connectedness, the claim is equivalent to
\[(y:Y) \rightarrow \left( r:\left( \mathtt{inl}(x_{0}) =_{X \sqcup^{Q}Y}^{}\mathtt{inr}(y) \right) \right) \rightarrow \mathtt{is\text{-}contr} \left(\left\| \mathtt{fib}_{ \mathtt{glue} } (r) \right\|_{m + n} \right).\]
Similar to \Cref{sec-algebra}, in order to apply path induction on \(r\), we
need to generalise the right-hand side of
\(\left( \mathtt{inl}(x_{0}) = \mathtt{inr}(y) \right)\) to an
arbitrary term \(p\) of \(X \sqcup^{Q}Y\). Therefore, we claim the
following lemma, which would imply \Cref{blakers-thm} by setting \(p\) to
\(\mathtt{inr}(y)\):

\begin{lemma}  There exists a dependent type
\[\mathtt{code}:\left( p:X \sqcup^{Q}Y \right) \rightarrow \left( \mathtt{inl}(x_{0}) =_{X \sqcup^{Q}Y}^{}p \right) \rightarrow \Type\]
such that

\begin{itemize}
\item
  \(\mathtt{code}(\mathtt{inr}(y))(r) \equiv \left\| \mathtt{fib}_{ \mathtt{glue} }  (r) \right\|_{m + n}\),
  and
\item
  for any \(p:X \sqcup^{Q}Y\) and any
  \(r:\left( \mathtt{inl}(x_{0}) =_{X \sqcup^{Q}Y}^{}p \right)\),
  \(\mathtt{code}(p)(r)\) is contractible.
\end{itemize}

\label{blakers-lemma} \end{lemma}

Additionally, we need the following relaxation:

\begin{lemma}  To prove \Cref{blakers-lemma}, we may assume that there exist some
fixed \(y_{0}:Y\) and
\(q_{00}:Q\left( x_{0} \right)\left( y_{0} \right)\). \end{lemma} \begin{proof} By
assumption, \(\sum_{y:Y}Q\left( x_{0} \right)(y)\) is \(m\)-connected,
so \(\left\| {\sum_{y:Y}Q\left( x_{0} \right)(y)} \right\|_{m}\) is
contractible and thus inhabited. The proposition that
\(\mathtt{code}(p)(r)\) is contractible is a mere proposition, and
thus is \emph{a fortiori} an \(m\)-type as \(m \geq - 1\). Thus, by
truncation induction, we may assume the inhabitant of
\(\left\| {\sum_{y:Y}Q\left( x_{0} \right)(y)} \right\|_{m}\) is of the
form \(\left| \left( y_{0},q_{00} \right) \right|_{m}\) for some
\(y_{0}:Y\) and \(q_{00}:Q\left( x_{0} \right)\left( y_{0} \right)\).
\end{proof}

In summary, our goal is to prove \Cref{blakers-lemma} while assuming some
fixed \(y_{0}:Y\) and
\(q_{00}:Q\left( x_{0} \right)\left( y_{0} \right)\). This is broken
into the definition of 	\texttt{code} and its
contractibility in the following two sections.

\section{Definition of
	\texttt{code}}\label{sec-def-code}

We begin with the wedge connectivity lemma. For two pointed types \(A\)
and \(B\), there is a canonical map
\(i:A \vee B \rightarrow A \times B\) given by
\(\lambda a.\left( a,b_{0} \right)\) and
\(\lambda b.\left( a_{0},b \right)\). By \Cref{connected-map-converse} and the
induction principle of wedge sums, the next lemma is an equivalent way
to state that if \(A\) is \(n\)-connected and \(B\) is \(m\)-connected,
then \(i\) is \((m + n)\)-connected.

\begin{lemma}[Wedge connectivity lemma, \protect{\cite[Lemma 8.6.2]{hott}}]
Suppose \(\left( A,a_{0} \right)\) and \(\left( B,b_{0} \right)\) are
\(n\)- and \(m\)-connected pointed types, respectively, with
\(n,m \geq - 1\)\footnote{This lemma works for \(m = - 1\) or
  \(n = - 1\) although \protect{\cite[Lemma 3.1.3]{hott}} requires \(n,m \geq 0\) \protect{\cite{kbh}}.}. Then given

\begin{itemize}
\item
  \(P:A \rightarrow B \rightarrow (n + m)\mathtt{\text{-}Type}\),
\item
  \(f:(a:A) \rightarrow P(a)\left( b_{0} \right)\),
\item
  \(g:(b:B) \rightarrow P\left( a_{0} \right)(b)\) and
\item
  \(p:f\left( a_{0} \right) =_{P\left( a_{0} \right)\left( b_{0} \right)}^{}g\left( b_{0} \right)\),
\end{itemize}
there is \(h:(a:A) \rightarrow (b:B) \rightarrow P(a)(b)\) with
homotopies \(q:h( - )\left( b_{0} \right) \sim f\) and
\(r:h\left( a_{0} \right) \sim g\) such that
\(p = {q\left( a_{0} \right)}^{- 1}\ct r\left( b_{0} \right)\).
\label{wedge-connect-lemma} \end{lemma}

In shorter terms, to define the dependent function
\(h:(a:A) \rightarrow (b:B) \rightarrow P(a)(b)\), it suffices to
consider respectively the cases when \(b\) coincides with \(b_{0}\) by
giving \(f\) and when \(a\) coincides with \(a_{0}\) by giving \(g\),
and then ensure the consistency when both cases happen simultaneously by
giving \(p\).

The following is an application of the wedge connectivity lemma, used in
defining 	\texttt{code}.

\begin{corollary} Let \(X,Y:\Type\) and
\(Q:X \rightarrow Y \rightarrow \Type\). Let \(x_{1}:X\),
\(y_{0}:Y\) and \(q_{10}:Q\left( x_{1} \right)\left( y_{0} \right)\).
Define type
\(A: \equiv \sum_{x_{0}:X}Q\left( x_{0} \right)\left( y_{0} \right)\)
and type
\(B: \equiv \sum_{y_{1}:Y}Q\left( x_{1} \right)\left( y_{1} \right)\).
Assume that \(A\) is \(n\)-connected and \(B\) is \(m\)-connected.
Define \(C:A \rightarrow B \rightarrow (m + n)\mathtt{\text{-}Type}\) by
\[\lambda\left( x_{0},q_{00} \right).\lambda\left( y_{1},q_{11} \right).\left\| {{\mathtt{fib}_{\mathtt{glue}}}_{x_{0},y_{1}}\left( \mathtt{glue}(q_{00})\ct{\mathtt{glue}(q_{10})}^{- 1}\ct\mathtt{glue}(q_{11}) \right)} \right\|_{m + n},\]
where \(q_{00}:Q\left( x_{0} \right)\left( y_{0} \right)\) and
\(q_{11}:Q\left( x_{1} \right)\left( y_{1} \right)\). Then there is a
dependent function
\(h:\left( \left( x_{0},q_{00} \right):A \right) \rightarrow \left( \left( y_{1},q_{11} \right):B \right) \rightarrow C\left( x_{0},q_{00} \right)\left( y_{1},q_{11} \right)\).
\label{sorry-no-name-lemma} \end{corollary}

\begin{proof} \(A\) is pointed by \(\left( x_{1},q_{10} \right)\) and \(B\)
is pointed by \(\left( y_{0},q_{10} \right)\). By \Cref{wedge-connect-lemma},
to obtain \(h\), it suffices to consider the following cases:

(1) When \(\left( y_{1},q_{11} \right)\) coincides with
\(\left( y_{0},q_{10} \right)\), for each
\(\left( x_{0},q_{00} \right):A\), we need a term of the type
\[C\left( x_{0},q_{00} \right)\left( y_{0},q_{10} \right) \equiv \left\| {\mathtt{fib}_{\mathtt{glue}}\left( \mathtt{glue}(q_{00})\ct{\mathtt{glue}(q_{10})}^{- 1}\ct\mathtt{glue}(q_{10}) \right)} \right\|_{m + n},\]
which is inhabited by
\(\left| \left( q_{00},\square \right) \right|_{m + n}\). (Here,
\(\square\) is the coherence path
\(\mathtt{glue}(q_{00}) = \mathtt{glue}(q_{00})\ct{\mathtt{glue}(q_{10})}^{- 1}\ct\mathtt{glue}(q_{10})\).)
(2) When \(\left( x_{0},q_{00} \right)\) coincides with
\(\left( x_{1},q_{10} \right)\), for each
\(\left( y_{1},q_{11} \right):B\), we need to a term of the type
\[C\left( x_{1},q_{10} \right)\left( y_{1},q_{11} \right) \equiv \left\| {\mathtt{fib}_{\mathtt{glue}}\left( \mathtt{glue}(q_{10})\ct{\mathtt{glue}(q_{10})}^{- 1}\ct\mathtt{glue}(q_{11}) \right)} \right\|_{m + n},\]
which is inhabited by
\(\left| \left( q_{11},\square \right) \right|_{m + n}\).

(3) Choices from (1) and (2) are consistent by coherence operations. \end{proof}

We are now ready to define 	\texttt{code}.

\begin{figure}[h]
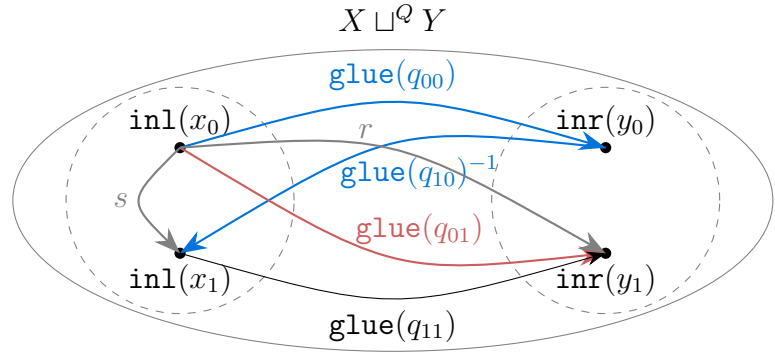

    \centering
    \figCode
    \caption{Diagram for the definition of
	\texttt{code}.}
    \label{fig-code}
\end{figure}

\begin{typee}
\(\mathtt{code}:\left( p:X \sqcup^{Q}Y \right) \rightarrow \left( \mathtt{inl}(x_{0}) =_{X \sqcup^{Q}Y}^{}p \right) \rightarrow \Type.\)
\label{prop-code} \end{typee} \begin{proof}See \Cref{fig-code}. When
\(p\) is \(\mathtt{inr}(y_{1})\), for
\(r:\left( \mathtt{inl}(x_{0}) =_{X \sqcup^{Q}Y}^{}\mathtt{inr}(y_{1}) \right)\),
as mentioned earlier,
\[\mathtt{code}(\mathtt{inr}(y_{1}))(r): \equiv \left\| {\sum_{q_{01}:Q\left( x_{0} \right)\left( y_{1} \right)}\mathtt{glue}(q_{01}) = r} \right\|_{m + n}.\]

When \(p\) is \(\mathtt{inl}(x_{1})\), for
\(s:\left( \mathtt{inl}(x_{0}) =_{X \sqcup^{Q}Y}^{}\mathtt{inl}(x_{1}) \right)\),
\[\mathtt{code}(\mathtt{inl}(x_{1}))(s): \equiv \left\| {\sum_{q_{10}:Q\left( x_{1} \right)\left( y_{0} \right)}\mathtt{glue}(q_{00})\ct{\mathtt{glue}(q_{10})}^{- 1} = s} \right\|_{m + n}.\]

\redDiamond When \(p\) `varies along'
\(\mathtt{glue}(q_{11}):\mathtt{inl}(x_{1}) = \mathtt{inr}(y_{1})\),
we need to find a path
\[\mathtt{apd}_{\mathtt{code}}\left( \mathtt{glue}(q_{11}) \right):\mathtt{code}(\mathtt{inl}(x_{1})) =_{\mathtt{glue}(q_{11})}^{\mathtt{code}}\mathtt{code}(\mathtt{inr}(y_{1}))\]
By function extensionality, it suffices to find for every
\(r:\left( \mathtt{inl}(x_{0}) = \mathtt{inr}(y_{1}) \right)\), a
path
\[{\mathtt{glue}(q_{11})}_{\ast}\left( \mathtt{code}(\mathtt{inl}(x_{1})) \right)(r) = \mathtt{code}(\mathtt{inr}(y_{1}))(r).\]
Starting from the left-hand side, we have a path \[\begin{aligned}
{\mathtt{glue}(q_{11})}_{\ast}\left( \mathtt{code}(\mathtt{inl}(x_{1})) \right)(r) & = {\mathtt{glue}(q_{11})}_{\ast}\left( \mathtt{code}(\mathtt{inl}(x_{1}))({\mathtt{glue}(q_{11})}_{\ast}^{- 1}(r)) \right) \\
 & = \mathtt{code}(\mathtt{inl}(x_{1}))(r\ct{\mathtt{glue}(q_{11})}^{- 1}),
\end{aligned}\] where the first equality is by \Cref{function-identity} and
the second by \Cref{transport-constant} and \Cref{transport-path}. Therefore, by
univalence, it remains to find a function
\begin{equation}
  \Phi:\mathtt{code}(\mathtt{inl}(x_{1}))(r\ct{\mathtt{glue}(q_{11})}^{- 1}) \rightarrow \mathtt{code}(\mathtt{inr}(y_{1}))(r),
  \label{func-phi}
\end{equation}
and show that \(\Phi\) is an equivalence. By the definition of
\(\mathtt{code}(\mathtt{inl}(x_{1}))\) and
\(\mathtt{code}(\mathtt{inr}(y_{1}))\) and truncation induction, it
suffices to find for each
\(q_{01}:Q\left( x_{1} \right)\left( y_{0} \right)\) with
\(\mathtt{glue}(q_{00})\ct{\mathtt{glue}(q_{10})}^{- 1} = r \,\ct \,{\mathtt{glue}(q_{11})}^{- 1}\),
a term of
\(\left\| {\sum_{q_{01}:Q\left( x_{0} \right)\left( y_{1} \right)}\mathtt{glue}(q_{01}) = r} \right\|_{m + n}\).
By coherence operations, we then eliminate \(r\), and it suffices to
find for each \(q_{10}:Q\left( x_{1} \right)\left( y_{0} \right)\), a
term of
\[\left\| {\sum_{q_{01}:Q\left( x_{0} \right)\left( y_{1} \right)}\mathtt{glue}(q_{01}) = \mathtt{glue}(q_{00})\ct{\mathtt{glue}(q_{10})}^{- 1}\ct\mathtt{glue}(q_{11})} \right\|_{m + n}.\]

We now invoke \Cref{sorry-no-name-lemma} (by using which the order of the
variables is effectively rearranged), and the desired term is given by
\(h\left( x_{0},q_{00} \right)\left( y_{1},q_{11} \right)\). We can
similarly define the function on the other direction and show that the
two functions are quasi-inverses, but the details are omitted. \end{proof}

\section{Contractibility of
	\texttt{code}}\label{sec-contra}

We now show that for any \(p:X \sqcup^{Q}Y\) and any
\(r:\left( \mathtt{inl}(x_{0}) =_{X \sqcup^{Q}Y}^{}p \right)\),
\(\mathtt{code}(p)(r)\) is contractible, by first giving its centre of
contraction.

\begin{typee}
\(\mathtt{centre}:\left( p:X \sqcup^{Q}Y \right) \rightarrow \left( r:\left( \mathtt{inl}(x_{0}) =_{X \sqcup^{Q}Y}^{}p \right) \right) \rightarrow \mathtt{code}(p)(r).\)
\end{typee} \begin{proof}By path induction, it suffices to assume
\(p \equiv \mathtt{inl}(x_{0})\) and
\(r \equiv \mathtt{refl}_{\mathtt{inl}(x_{0})}\). Then
\[\mathtt{code}(\mathtt{inl}(x_{0}))(\mathtt{refl}_{\mathtt{inl}(x_{0})}) \equiv \left\| {\sum_{q:Q\left( x_{0} \right)\left( y_{0} \right)}\mathtt{glue}(q_{00})\ct{\mathtt{glue}(q)}^{- 1} = \mathtt{refl}_{\mathtt{inl}(x_{0})}} \right\|_{m + n},\]
which is inhabited by
\(\left| \left( q_{00},\square \right) \right|_{m + n}\). Then
	\texttt{centre} is explicitly given by
\[\mathtt{centre}(p)(r): \equiv \mathtt{transport}^{\hat{\mathtt{code}}}\left( \mathtt{pair}\hspace{0pt}^{=}(r,t) \right)\left( \left| \left( q_{00},\square \right) \right|_{m + n} \right)\]
by \Cref{transport-path-induction}, where
\(t:r_{\ast}\left( \mathtt{refl}_{x_{0}} \right) = r\) by
\Cref{transport-path} and
\[\hat{\mathtt{code}}:\sum_{p:X \sqcup^{Q}Y}\left( \mathtt{inl}(x_{0}) =_{X \sqcup^{Q}Y}^{}p \right) \rightarrow \Type\]
is the uncurried form of 	\texttt{code}. \end{proof}

Before we proceed, the next lemma allows us to extract some valuable
information from the hard-earned definition of
	\texttt{code} for calculating a certain
	\texttt{transport} function.

\begin{lemma}[\protect{\cite[Lemma 8.6.10]{hott}}]
 Let \(A:\Type\),
\(B:A \rightarrow \Type\),
\(C:(a:A) \rightarrow B(a) \rightarrow \Type\),
\(a_{1},a_{2}:A\), \(m:a_{1} =_{A}^{}a_{2}\), and
\(b_{2}:B\left( a_{2} \right)\). Denote
\(b_{1}: \equiv m_{\ast}^{- 1}\left( b_{2} \right):B\left( a_{1} \right)\).
Also denote the uncurried form of \(C\) as
\(\hat{C}:\sum_{a:A}B(a) \rightarrow \Type\). Then the function
\[\mathtt{transport}^{\hat{C} } \left( \mathtt{pair}\hspace{0pt}^{=}(m,t) \right):C\left( a_{1} \right)\left( b_{1} \right) \rightarrow C\left( a_{2} \right)\left( b_{2} \right),\]
where
\(t:m_{\ast}\left( b_{1} \right) =_{B\left( a_{2} \right)}^{}b_{2}\) is
an obvious path, is equal to the equivalence obtained by applying
\(\mathtt{idtoeqv}\) to the composite of paths
\[C\left( a_{1} \right)\left( b_{1} \right) = m_{\ast}\left( C\left( a_{1} \right)\left( m_{\ast}^{- 1}\left( b_{2} \right) \right) \right) = m_{\ast}\left( C\left( a_{1} \right) \right)\left( b_{2} \right) = C\left( a_{2} \right)\left( b_{2} \right),\]
where the first equality is by definition and by \Cref{transport-constant},
the second is by \Cref{function-identity}, and the third is by
\(\mathtt{happly}(\mathtt{apd}_{C}(m))\left( b_{2} \right)\). 
\label{ABC-lemma}
\end{lemma}

\begin{figure}[H]
    \centering
    \figAbc
    \caption{Diagram for \Cref{ABC-lemma}.}
    \caption*{\small The blue
\(m_{\ast}\) represents \(\mathtt{transport}^{B}(m)\), the green
\(m_{\ast}\) represents
\(\mathtt{transport}^{\lambda a.\Type}(m)\) and the red
\(m_{\ast}\) represents \(\mathtt{transport}^{C}(m)\). The arrow
between \(m_{\ast}\left( C\left( a_{1} \right) \right)\) and
\(C\left( a_{2} \right)\) represents the dependent path
\(\mathtt{apd}_{C}(m)\).}
    \label{fig:enter-label}
\end{figure}

\redDiamond We apply
\Cref{ABC-lemma} with \(A \equiv X \sqcup^{Q}Y\),
\(B \equiv (p:A) \rightarrow \left( \mathtt{inl}(x_{0}) = p \right)\),
\(C \equiv \mathtt{code}\), \(a_{1} \equiv \mathtt{inl}(x_{1})\),
\(a_{2} \equiv \mathtt{inr}(y_{1})\),
\(m \equiv \mathtt{glue}(q_{11})\), and
\(b_{2} \equiv \mathtt{glue}(q_{01})\). Then
\(b_{1} \equiv m_{\ast}^{- 1}\left( b_{2} \right) = \mathtt{glue}(q_{01})\ct{\mathtt{glue}(q_{11})}^{- 1}\)
by \Cref{transport-path}. Then \Cref{ABC-lemma} says that to calculate the function
\[T: \equiv \mathtt{transport}^{ \hat{\mathtt{code}} }\left( \mathtt{pair}\hspace{0pt}^{=}(\mathtt{glue}(q_{11}),t) \right):\mathtt{code}(\mathtt{inl}(x_{1}))(b_{1}) \rightarrow \mathtt{code}(\mathtt{inr}(y_{1}))(b_{2}),\]
the major work lies in recovering our definition of
\(\mathtt{apd}_{\mathtt{code}}\left( \mathtt{glue}(q_{11}) \right)\),
which is a part of the definition of \(\mathtt{code}\). Noticing that
the type of \(T\) matches the type of \(\Phi\) (Function~\ref{func-phi}), to calculate
\(T\), we can utilise our definition of \(\Phi\), with \(r\) replaced
with \(b_{2} \equiv \mathtt{glue}(q_{01})\). The definition of
\(\Phi\) relies on \Cref{sorry-no-name-lemma} and we know explicitly the
output under the two special cases in the proof of \Cref{sorry-no-name-lemma}.
In particular, suppose that \(\left( x_{0},q_{00} \right)\) coincides
with \(\left( x_{1},q_{10} \right)\) and further \(q_{01}\) coincides
with \(q_{11}\). By case (2) in the proof of \Cref{sorry-no-name-lemma} and
the full statement of \Cref{wedge-connect-lemma}, \(h\) gives the term
\(\left| \left( q_{01},\square \right) \right|_{m + n}\) (up to a path),
and thus so does \(\Phi\) and \(T\).

Finally, we show that each \(\mathtt{code}(p)(r)\) is contractible,
written as: \begin{typee}
\(\left( p:X \sqcup^{Q}Y \right) \rightarrow \left( r:\left( \mathtt{inl}(x_{0}) =_{X \sqcup^{Q}Y}^{}p \right) \right) \rightarrow \left( c:\mathtt{code}(p)(r) \right) \rightarrow \left( \mathtt{centre}(p)(r) = c \right).\)
\label{code-contr} \end{typee}
\begin{remark}\redDiamond It is tempting to assume by path
induction that \(p\) is \(\mathtt{inl}(x_{0})\) and \(r\) is
\(\mathtt{refl}_{\mathtt{inl}(x_{0})}\). The desired path is in
\(\mathtt{code}(p)(r)\) and thus is an \((m + n - 1)\)-type (and
\emph{a fortiori} an \((m + n)\)-type), so by truncation induction, we
assume that \(c\) is
\(\left| \left( q,\ell_{1} \right) \right|_{m + n}\), where
\(q:Q\left( x_{0} \right)\left( y_{0} \right)\) and
\(\ell_{1}:\mathtt{glue}(q_{00})\ct{\mathtt{glue}(q)}^{- 1} = \mathtt{refl}_{\mathtt{inl}(x_{0})}\).
By the definition of \(\mathtt{centre}\), we need a path
\(\left| \left( q_{00},\square \right) \right|_{m + n} = \left| \left( q,\ell_{1} \right) \right|_{m + n}.\)
Now we can only attempt to prove \(q_{00} = q\), which is unattainable,
as \(\ell_{1}\) only indicates
\(\mathtt{glue}(q_{00}) = \mathtt{glue}(q)\). \end{remark}

\begin{proof} By \Cref{strong-path-induction}, it suffices to assume that \(p\)
is \(\mathtt{inr}(y_{0})\) and \(r\) is \(\mathtt{glue}(q_{00})\).
We need a re-generalisation for the \(y\)`s and claim the following:
\[\left( y_{1}:Y \right) \rightarrow \left( r:\mathtt{inl}(x_{0}) = \mathtt{inr}(y_{1}) \right) \rightarrow \left( c:\mathtt{code}(\mathtt{inr}(y_{1}))(r) \right) \rightarrow \left( \mathtt{centre}(\mathtt{inr}(y_{1}))(r) = c \right).\]
By truncation induction, it suffices to assume \(c\) is
\(\left| \left( q_{01},\ell_{2} \right) \right|_{m + n}\) for some
\(q_{01}:Q\left( x_{0} \right)\left( y_{1} \right)\) and
\(\ell_{2}:\mathtt{glue}(q_{01}) = r\). Noticing that since \(r\) is a
free point in the type
\(\mathtt{inl}(x_{0}) = \mathtt{inr}(y_{1})\), we use path induction
on \(\ell_{2}\) to assume that \(r\) is \(\mathtt{glue}(q_{01})\) and
\(\ell_{2}\) is \(\mathtt{refl}_{\mathtt{glue}(q_{01})}\). It then
remains to show that, for any \(y_{1}:Y\) and any
\(q_{01}:Q\left( x_{0} \right)\left( y_{1} \right)\), we have
\[\mathtt{centre}(\mathtt{inr}(y_{1}))(\mathtt{glue}(q_{01})) = \left| \left( q_{01},\mathtt{refl}_{\mathtt{glue}(q_{01})} \right) \right|_{m + n}.\]
By the definition of \(\mathtt{centre}\), the left-hand side is
\[\mathtt{transport}^{\hat{\mathtt{code}}}\left( \mathtt{pair}\hspace{0pt}^{=}(\mathtt{glue}(q_{01}),t) \right)\left( \left| \left( q_{00},\square \right) \right|_{m + n} \right).\]
This is precisely the case we discussed after \Cref{ABC-lemma} (since now
\(x_{0}\), \(q_{00}\) and \(q_{01}\) coincide with \(x_{1}\),
\(q_{10}\), \(q_{11}\)), thus it is indeed equal to
\(\left| \left( q_{01},\square \right) \right|_{m + n}\). \end{proof} We have
hence completed the proof of \Cref{blakers-lemma} and thus \Cref{blakers-thm}.

\section{Consequences}\label{sec-consequence}

The contractibility of \(\mathtt{code}\) in fact implies more than
\Cref{blakers-thm}, which only uses the case
\(\mathtt{code}(\mathtt{inr}(y_{1}))\). Our definition for the other
case, \(\mathtt{code}(\mathtt{inl}(x_{1}))\), immediately implies
the following: \begin{proposition}\redDiamond Assume the same as
in \Cref{blakers-thm}. Also let \(x_{0}:X\), \(y_{0}:Y\) and
\(q_{00}:Q\left( x_{0} \right)\left( y_{0} \right)\). Then for any
\(x_{1}:X\), the map
\[\tau:\lambda q_{10}.\mathtt{glue}(q_{00})\ct{\mathtt{glue}(q_{10})}^{- 1}:Q\left( x_{1} \right)\left( y_{0} \right) \rightarrow \left( \mathtt{inl}(x_{0}) = \mathtt{inl}(x_{1}) \right)\]
is \((m + n)\)-connected. \label{dual-blakers} \end{proposition}
\begin{proof} For any
\(s:\left( \mathtt{inl}(x_{0}) = \mathtt{inl}(x_{1}) \right)\),
\(\left\| {\mathtt{fib}_{\tau}(s)} \right\|_{m + n}\) matches our
definition for \(\mathtt{code}(\mathtt{inl}(x_{1}))(s)\), which is
contractible by \Cref{code-contr}. \end{proof} Curiously, it is this dual case,
instead of \Cref{blakers-thm}, that directly gives rise to the Freudenthal
suspension theorem as below.
\begin{theorem}[Freudenthal suspension theorem]
  Let \(C\) be an \(n\)-connected type with \(n \geq - 1\) and \(x_{0}:C\). 
  Then the map \(\sigma:C \rightarrow \Omega(\Sigma C)\), given by
  \(\sigma: \equiv \lambda x.\mathtt{merid}(x)\ct{\mathtt{merid}(x_{0})}^{- 1}\),
  is \(2n\)-connected.
\label{freu} \end{theorem} \begin{proof}
\redDiamond We rewrite the suspension using the generalised
pushout. The suspension \(\Sigma C\) is the pushout of the diagram
\(\mathbb{1}_{N}\overset{f}{\leftarrow}C\overset{g}{\rightarrow}\mathbb{1}_{S}\),
so \(X \equiv \mathbb{1}_{N}\), \(Y \equiv \mathbb{1}_{S}\) and \(f,g\)
are two trivial maps. Then in this case,
\(Q\left( \star_{N} \right)\left( \star_{S} \right)\) is equivalent to
\(C\) and under this equivalence the map \(\mathtt{glue}_{NS}\) is
identified with \(\mathtt{merid}\). The connectivity assumptions are
easily seen to be satisfied, so by \Cref{dual-blakers}, the map
\(\sigma^{\prime}: \equiv \lambda x.\mathtt{merid}(x_{0})\ct{\mathtt{merid}(x)}^{- 1}:C \rightarrow \Omega(\Sigma C)\)
is \(2n\)-connected. Hence, for any \(s:\Omega(\Sigma C)\), the fibre
\(\mathtt{fib}_{\sigma^{\prime}}(s)\) is \(2n\)-connected. Notice that by
the definition of fibres and coherence operations, \[\begin{aligned}
\mathtt{fib}_{\sigma^{\prime}}  \left( s^{- 1} \right) & \equiv \sum_{x:C}\left( \mathtt{merid}(x_{0})\ct{\mathtt{merid}(x)}^{- 1} = s^{- 1} \right) \\
 & \simeq \sum_{x:C}\left( \mathtt{merid}(x)\ct{\mathtt{merid}(x_{0})}^{- 1} = s \right) \equiv \mathtt{fib}_{\sigma} (s),
\end{aligned}\] so \(\sigma\) is also \(2n\)-connected. \end{proof}

\begin{corollary}[Freudenthal equivalence, \protect{\cite[Corollary 8.6.14]{hott}}]
  Suppose that type \(X\) is \(n\)-connected and pointed
with \(n \geq - 1\). Then
\(\left\| X \right\|_{2n} \simeq \left\| {\Omega(\Sigma X)} \right\|_{2n}\).
\label{freu-equiv} \end{corollary} \begin{proof}By \Cref{freu} and
\Cref{connected-induce-equivalence}.\end{proof} We can now complete
\Cref{tab-homotopy-groups} as promised. 
\begin{corollary}[Stability for spheres]
  If \(k \leq 2n - 2\), then
  \(\pi_{k + 1}\left( {\mathbb{S}}^{n + 1} \right) \simeq \pi_{k}\left( {\mathbb{S}}^{n} \right)\).
\label{stable-sphere} \end{corollary} 
\begin{proof} By \Cref{s-n-connected},
\({\mathbb{S}}^{n}\) is \((n - 1)\)-connected. By \Cref{freu-equiv} and by the
definition of \(n\)-spheres,
\(\left\| {\mathbb{S}}^{n} \right\|_{2(n - 1)} \simeq \left\| {\Omega(\Sigma{\mathbb{S}}^{n})} \right\|_{2(n - 1)} \equiv \left\| {\Omega({\mathbb{S}}^{n + 1})} \right\|_{2(n - 1)}\).
Since \(k \leq 2(n - 1)\), by \Cref{trun-trun},
\(\left\| {\mathbb{S}}^{n} \right\|_{k} \simeq \left\| {\Omega({\mathbb{S}}^{n + 1})} \right\|_{k}\).
We then calculate \[\begin{aligned}
\pi_{k + 1}\left( {\mathbb{S}}^{n + 1} \right) & \equiv \left\| {\Omega^{k + 1}\left( {\mathbb{S}}^{n + 1} \right)} \right\|_{0} \equiv \left\| {\Omega^{k}\left( \Omega\left( {\mathbb{S}}^{n + 1} \right) \right)} \right\|_{0} \\
 & \simeq \Omega^{k}\left\| {\Omega\left( {\mathbb{S}}^{n + 1} \right)} \right\|_{k} \simeq \Omega^{k}\left\| {\mathbb{S}}^{n} \right\|_{k} \simeq \left\| {\Omega^{k}\left( {\mathbb{S}}^{n} \right)} \right\|_{0} \equiv \pi_{k}\left( {\mathbb{S}}^{n} \right),
\end{aligned}\] where we use \Cref{truncation-omega-commute} twice to exchange
truncation and \(\Omega\). \end{proof} \begin{corollary}
\(\pi_{n}\left( {\mathbb{S}}^{n} \right) \simeq {\mathbb{Z}}\) for
\(n \geq 1\). \label{pinsn-thm} \end{corollary} \begin{proof}
\(\pi_{1}\left( {\mathbb{S}}^{1} \right) \simeq {\mathbb{Z}}\) by \Cref{pk-s1}
and \(\pi_{2}\left( {\mathbb{S}}^{2} \right) \simeq {\mathbb{Z}}\) by
\Cref{pi2-pi3}. When \(n \geq 2\), \(n \leq 2(n - 1)\), so by induction on
\(n\) and \Cref{stable-sphere}, the result follows. \end{proof}

\begin{corollary}
\(\pi_{3}\left( {\mathbb{S}}^{2} \right) \simeq {\mathbb{Z}}\). \label{pi3s2} \end{corollary} \begin{proof}
\(\pi_{3}\left( {\mathbb{S}}^{2} \right) \simeq \pi_{3}\left( {\mathbb{S}}^{3} \right) \simeq {\mathbb{Z}}\)
by \Cref{pi2-pi3} and \Cref{pinsn-thm}. \end{proof}

\chapter{Conclusion and Further Discussions}

\begin{table}[ht]
\centering
\caption{Main results and their locations in this dissertation.}
\label{tab-main-results}
\renewcommand{\arraystretch}{1.4}
\begin{tabular}{>{\centering\arraybackslash}m{0.45\textwidth} >{\centering\arraybackslash}m{0.45\textwidth}}
\toprule
\textbf{Result} & \textbf{Location(s)} \\
\midrule
Homotopy groups of $n$-spheres  in \Cref{tab-homotopy-groups} & Corollaries~\ref{pk-s1}, \ref{pk-sn-k-less}, \ref{pi2-pi3}, \ref{pinsn-thm}, \ref{pi3s2} \\
Fibre sequence of a pointed map & \Cref{fibre-sequence-2} \\
Long exact sequence of a pointed map & \Cref{fibre-les} \\
Hopf construction & \Cref{hopf-construction} \\
Hopf fibration & \Cref{hopf-fibration} \\
Blakers--Massey theorem & \Cref{blakers-thm} \\
Freudenthal suspension theorem & \Cref{freu} \\
Stability for spheres & \Cref{stable-sphere} \\
\bottomrule
\end{tabular}
\end{table}

The main results in this dissertation are listed in \Cref{tab-main-results}.
We briefly evaluate our synthetic approach to homotopy theory:

\begin{itemize}
\item
  \emph{Formal abstraction} allows us to focus on the homotopy
  properties without considering the underlying set-theoretic
  structures. For example, unlike in classical topology, the universal
  cover of \({\mathbb{S}}^{1}\) (see \Cref{s1-lemma}) does not involve real
  numbers \(\mathbb{R}\) (a helix projecting onto the circle
  \protect{\cite[Theorem 1.7]{hatcher}}) but only a fibration with fibre
  \(\mathbb{Z}\), a primitive notion expressed by a dependent type.
\item
  All notions are \emph{homotopy-invariant}. For example, our formation
  of the Blakers--Massey theorem \Cref{blakers-thm} replaces set intersection
  and union in classical text \protect{\cite[Theorem 4.23]{hatcher}} with pushouts,
  leading to a more general statement.
\item
  \emph{Proof relevance} has enabled us to effectively reuse pieces of
  previous results, such as the application of \Cref{ABC-lemma} in
  \Cref{sec-contra}.
\item
  The \emph{rule-based} nature of type theory has rendered many proofs
  routine applications of induction principles and existing lemmas.
  Whilst this allows for mechanised reasoning, occasionally proofs lack
  geometric intuition and appear overly technical. Nevertheless, there
  are methods, such as \protect{\cite[Section 8.1.5]{hott}} for
  \(\pi_{k}({\mathbb{S}}^{1}\)) and \protect{\cite[Section 4]{blakers}} for the
  Blakers--Massey theorem, that offer easier homotopy-theoretic
  interpretations.
\end{itemize}

These further topics are related to our work:

\begin{itemize}
\item
  The proof of the Blakers--Massey theorem \protect{\cite{blakers}} has inspired a
  `reverse engineered' classical proof \protect{\cite{rezk}} and a generalised version
  valid in an arbitrary higher topos and with respect to an arbitrary
  \emph{modality} \protect{\cite{anel}}.
\end{itemize}

\begin{itemize}
\item
  The Blakers--Massey theorem is used in proving that there is a natural
  number \(n\) such that
  \(\pi_{4}\left( {\mathbb{S}}^{3} \right) \simeq {\mathbb{Z}}/n{\mathbb{Z}}\)
  \protect{\cite[Proposition 3.4.4]{brunerie}}.
\end{itemize}

\begin{itemize}
\item
  There is a fibration over \({\mathbb{S}}^{4}\) with fibre
  \({\mathbb{S}}^{3}\) and total space \({\mathbb{S}}^{7}\) by giving an
  H-space structure on \({\mathbb{S}}^{3}\) \protect{\cite[Theorem 4.10]{cayley}}.
  In particular, \(\pi_{7}\left( {\mathbb{S}}^{4} \right)\) has an
  element of infinite order \protect{\cite[Corollary 4.11]{cayley}}.
\end{itemize}

We also state some speculations based on current work:

\begin{itemize}
\item
  \Cref{sec-blakers} has left the proofs on coherence operations unattended.
  As remarked by \protect{\cite[p. 23]{brunerie}}, in general, these proofs can be
  done via recursive application of path inductions. Here we speculate
  that there may be automated methods or some `meta-theorem' that
  reduce the need to spell out these proofs.
\item
  Our diagrams, combined with the topological interpretations of HoTT,
  suggest the potential for a systematic approach to HoTT-based
  graphical reasoning. For example, proof assistants may be equipped
  with a graphical interface, where the application of rules is
  visualised through shape transformations. This could potentially offer
  more intuition and simplify mechanised reasoning.
\end{itemize}

\include{conclusions}

\appendix
\chapter{The `hub and spoke' construction of \texorpdfstring{$n$}{n}-truncations}
\label{app:hub}
\newcommand{\figTrun}{
\begin{tikzpicture}[
    label/.style={rectangle, inner sep=4pt, outer sep=0pt},
]
\draw[black] (-2,0) circle (1.5cm) node[label, above=1.5cm, black] {$\mathbb{S}^{n+1}$};
\draw[black] (5,0) circle (3cm) node[label, above=3cm, black] {$\left\| A \right \| _n$};
\draw[gray, dashed] (5,0.6) circle (1.6cm);
\fill (-2,-1.5) circle (2pt) node[label, below] {$\mathtt{base}$};
\fill (-3.5,0) circle (2pt) node[label, left] {$x$};
\fill (5,-1) circle (2pt) node[label, below left] {$r(\mathtt{base})$};
\fill (3.4,0.6) circle (2pt) node[label, left] {$r(x)$};
\fill (5,-2.4) circle (2pt) node[label, below] {$b$};
\fill (5.5,2.11) circle (2pt) node[label, above] {$h(r)$};
\draw[-{Stealth[scale=1.5]}] (3.4,0.6) -- node[above left, xshift=8pt] {$s(r)(x)$} (5.5,2.11);
\draw[-{Stealth[scale=1.5]}] (5,-1) -- node[right] {$s(r)(\mathtt{base})$} (5.5,2.11);
\draw[-{Stealth[scale=1.5]}] (5,-1) -- node[right] {$r_*$} (5,-2.4);
\draw[-{Straight Barb[scale=1.5]}] (-0.5,0) -- node[above] {$r$} (2, 0);
\end{tikzpicture}
}

\begin{definition}
Let \(A:\mathtt{Type}\) and \(n \geq - 1\). Define
\(\left\| A \right\|_{n}\), the \textbf{\(n\)-truncation} of \(A\), as
the higher inductive type generated by two point constructors
\(| - |_{n}:A \rightarrow \left\| A \right\|_{n}\) and
\(h:\left( {\mathbb{S}}^{n + 1} \rightarrow \left\| A \right\|_{n} \right) \rightarrow \left\| A \right\|_{n}\)
and a path constructor
\[s:\left( r:\left( {\mathbb{S}}^{n + 1} \rightarrow \left\| A \right\|_{n} \right) \right) \rightarrow \left( x:{\mathbb{S}}^{n + 1} \right) \rightarrow \left( r(x) =_{\left\| A \right\|_{n}}^{}h(r) \right).\]
\end{definition}

The rationale behind \(h\) and \(s\) (`hub and spoke') in this definition will be clear in
the following proof, visualised by \Cref{fig-trun}.

\begin{figure}[H]
    \centering
    \figTrun
    \caption{Diagram for \Cref{truncation-n-type}.}
    \caption*{\small \(n\) is thought of as \(0\) and the dashed circle
represents the `image' of \({\mathbb{S}}^{n + 1}\) under \(r\).}
    \label{fig-trun}
\end{figure}

\begin{lemma}[\protect{\cite[Lemma 7.3.1]{hott}}]
    For any \(A:\Type\) and
\(n \geq - 1\), \(\left\| A \right\|_{n}\) is an \(n\)-type.
\label{truncation-n-type}
\end{lemma}
\begin{proof}
Take any \(b:\left\| A \right\|_{n}\) and let
\(\left\| A \right\|_{n}\) be pointed by \(b\). By \Cref{omega-contractible},
it suffices to prove that \(\Omega^{n + 1}\left\| A \right\|_{n}\) is
contractible. By a corollary of the `suspension-loop adjunction' \cite[Lemma 6.5.4]{hott}, it then suffices to prove that the
type
\(M: \equiv \mathtt{Map}\hspace{0pt}_{\star}\left( {\mathbb{S}}^{n + 1},\left\| A \right\|_{n} \right)\)
is contractible. The constant function \(c_{b}: \equiv \lambda x.b\) is
a pointed map of type \(M\) witnessed by \(\mathtt{refl}_{b}\), and we
claim it to be the centre of contraction, i.e., for any pointed function
\(r:{\mathbb{S}}^{n + 1} \rightarrow_{\star}\left\| A \right\|_{n}\)
witnessed by \(r_{\star}:r\left( \mathtt{base} \right) = b\), we have
\(\left( c_{b},\mathtt{refl}_{b} \right) =_{M}^{}\left( r,r_{\star} \right)\).
By \Cref{sigma-path}, it suffices to find a path \(q:c_{b} = r\) such that
\(q_{\ast}\left( \mathtt{refl}_{b} \right) = r_{\star}\). Using
function extensionality, to find \(q\), it suffices to find a path
\(q_{x}:b = r(x)\) for any \(x:{\mathbb{S}}^{n + 1}\). Now using the
constructors we have a point \(h (r) :\left\| A \right\|_{n}\) along with
two paths \(s(r) (x) :r(x) = h(r)\) and
\( s(r) \left( \mathtt{base} \right) :r\left( \mathtt{base} \right) = h(r)\).
Thus the desired path \(q_{x}\) is given by
\[r_{\star}^{- 1}\, \ct \, s(r) \left( \mathtt{base} \right)  \, \ct \,s{(r)(x)}^{- 1}.\]
Finally, by \Cref{identity-path-1},
\(q_{\ast}\left( \mathtt{refl}_{b} \right) = r_{\star}\) is equivalent
to
\(q_{\mathtt{base}}^{- 1}\, \ct \,\mathtt{refl}_{b}\, \ct \,\mathtt{refl}_{b} = r_{\star}\),
which holds by the definition of \(q_{\mathtt{base}}\) and coherence
operations. 
\end{proof}

\addcontentsline{toc}{chapter}{Bibliography}
\bibliographystyle{alpha}  
\bibliography{main}        

\end{document}